\documentclass[a4paper]{article}
\usepackage{hyperref}
\usepackage{amsmath, amsthm, amsfonts, amssymb}
\usepackage{upgreek}

\theoremstyle{plain}
\newtheorem{thm}{Theorem}[section]
\newtheorem{coro}[thm]{Corollary}
\newtheorem{prop}[thm]{Proposition}

\newtheorem{lemm}[thm]{Lemma}

\theoremstyle{definition}
\newtheorem{deff}[thm]{Definition}
\newtheorem{examp}[thm]{Example}

\theoremstyle{remark}
\newtheorem{rema}[thm]{Remark}
\newtheorem*{conv}{Convention}

\renewcommand\thefootnote{\alph{footnote}}

\newcommand\legendre[2]{\genfrac{(}{)}{}{}{#1}{#2}}

\title{Finite quadratic modules and lattices}
\author{Xiao-Jie Zhu}

\begin{document}
\maketitle
\let\oldthefootnote=\thefootnote
\let\thefootnote\relax\footnotetext{\textsl{2010 Mathematics Subject Classification.} 11E12 (primary), 16D70 (secondary).}
\footnotetext{\textsl{Key words and phrases.} Finite quadratic modules, lattices, discriminant modules.}
\footnotetext{This work is supported by the Fundamental Research Funds for the Central Universities of China (Grant No. 22120180508).}
\footnotetext{\textsl{Address.} School of Mathematical Sciences, Tongji University, 1239 Siping Road, Shanghai, P.R. China}
\footnotetext{\textsl{E-mail.} zhuxiaojiemath@outlook.com}
\let\thefootnote=\oldthefootnote

\begin{abstract}
We give a new proof of the fact that any finite quadratic module can be decomposed into indecomposable ones. For any indecomposable finite quadratic module, we construct a lattice, and a positive definite lattice, both of which are of the least rank, whose discriminant module is the given one. The resulting lattices are given by their Gram matrices explicitly.
\end{abstract}

\tableofcontents

\section{Introduction}
\label{sec:Introduction}
A finite quadratic module is a finite abelian group equipped with a nondegenerate quadratic map assuming values in the quotient $\mathbb{Q}/\mathbb{Z}$, and a lattice is a free $\mathbb{Z}$-module of finite rank equipped with a nondegenerate symmetric bilinear form taking values in the reals. The results of this paper concern these two concepts. Our first result is a new proof of the structure theorem of finite quadratic modules, which says any such module can be decomposed as an orthogonal direct sum of three kinds of indecomposable modules. Our proof involves only some theory of finite abelian groups, not using methods of $p$-adic fields and $p$-adic lattices. To our best knowledge, the first proof of this theorem stated in such form appears in \cite{Stromberg2013}, where the author used classical results from the theory of quadratic forms over $p$-adic numbers. We must mention that, Wall has actually proved this in \cite[\S 6]{Wall1963}, but he did not state it explicitly. Our second result is that, we find an even lattice of the least rank whose discriminant module is any given indecomposable finite quadratic module. These lattices are presented by their Gram matrices. In this direction, Wall is the first person to prove any finite quadratic module is the discriminant module of some  even lattice. See \cite{Wall1963}. Comparing to Wall's result, our result has the advantage of giving the lattice explicitly for any indecomposable module. Our third result is analogous to the second result, which gives, for any indecomposable module, a positive definite even lattice instead of a general even lattice, still of the least rank. The motivation of such constructions is the need of providing concrete examples of Jacobi forms of lattice index.

We now briefly review relevant results not mentioned above in the literature. The importance of the concept and theory of finite quadratic modules, to our best knowledge, lies in two facts. One fact is that the concept of finite quadratic modules is a useful tool to study non-unimodular integral even lattices. This was shown by Nikulin in \cite{Nikulin1980}. More specifically, Nikulin gives a necessary and sufficient condition for the existence of even lattices with given inertia and discriminant module, which generalizes a theorem of John Milnor dealing with unimodular lattices. See \cite[Theorem 1.1.1]{Nikulin1980} and \cite[Theorem 1.10.1]{Nikulin1980}. The other fact is that, to any finite quadratic module, one can associate a Weil representation of $\mathrm{Mp}_2(\mathbb{Z})$, the metaplectic cover of the modular group. (For background materials on Weil representations, see \cite{Gelbart1976}, or the original paper of Weil \cite{Weil1964}.) The Weil representations of this kind are very interesting objects in Representation Theory, since they include all representations of the modular group (or its double cover) whose kernels are congruence subgroups. See \cite{Skoruppa2020} for a thorough discussion and a proof. Also see \cite{Skoruppa2008}. They are also interesting in the area of modular forms, since any modular form of level $N \in \mathbb{Z}_{\ge 1}$ with the multiplier system given by a Dirichlet character can be lifted to a modular form for a Weil representation associated to some finite quadratic module. One can find an explicit formula of this lift in \cite[Theorem 5.7]{Scheithauer2009}. Also, this kind of Weil representations have important applications in the theory of automorphic forms on orthogonal groups. The key concept of this application is that of a Bocherds' singular theta lift, which maps a modular form for some Weil representation to an automorphic form on some orthogonal group. See \cite{Borcherds1998} for this theory. This concludes the discussion of the importance of the concept and theory of finite quadratic modules. Besides the above discussion, we suggest another good refenrence on finite quadratic modules and relevant topics, the book \cite{CS2017} (more exactly, \S 14.5).

We list the notational convention. As usual, the symbol $\mathbb{Z}$, $\mathbb{Q}$, $\mathbb{R}$, and $\mathbb{C}$ denote the ring of integers, of rationals, of reals, and of complex numbers respectively. We use $\mathbb{Z}_{\ge n}$ to denote the set of integers greater than or equal to $n$. For two integers $a$ and $b$, the notation $(a,b)$ denotes the greatest common divisor of these two numbers, and $\legendre{a}{b}$ denotes the Kronecker's extension of the Legendre symbol. For a definition of Kronecker symbol $\legendre{a}{b}$, see \cite[\S 3.4]{CS2017}. For a ring $R$, we use $R^\times$ to denote its group of units, and use $\mathrm{GL}_n(R)$ to denote the group of invertible $n \times n$ matrices with entries in $R$, and $\mathrm{SL}_n(R)$ the subgroup of matrices whose determinant equals 1 in $\mathrm{GL}_n(R)$. For an $n \times n$ matrix $A$, the symbol $A^{\mathrm{T}}$ denotes its transpose, and $A^{-1}$ its inverse. The indentity matrix of size $n$ would be denoted by $\mathrm{I}_n$. The symbol $\mathop{\mathrm{diag}}\left(a_1,a_2, \dots, a_n\right)$ denotes the $n \times n$ diagonal matrix whose entries on the main diagonal are $a_1$, $a_2$, \dots, and $a_n$. The symbol $\delta_{ij}$ denotes Kronecker's $\delta$, i.e., the function assumes the value $1$ if $i=j$ and the value $0$ if $i \neq j$. For a set $S$, by $\vert S \vert$ we mean its cardinality. For a finite abelian group $G$, by $\widehat{G}$ we mean its dual group, i.e., the group of all complex linear characters on $G$ with the usual multiplication of functions. If $\sigma$ is a homomorphism of two algebraic structures with an underlying abelian group structure (written additively), then $\ker(\sigma)$ denotes the inverse image of the zero element. Other notations can be found where we give definitions.

The structure of this paper is as follows. In Section \ref{sec:Billinear map modules and quadratic map modules}, we review some basic facts concerning modules, bilinear maps, quadratic maps, orthogonal sums, matrices of bilinear forms, how to take quotients, and the concept of nondegeneracy. This is necessary for our purpose: we shall use these basic facts in our proof of the structure theorem of finite quadratic modules and also in our construction of a lattice whose discriminant module is any given finite quadratic module. In Section \ref{sec:Lattices and finite quadratic modules}, after a review of basic facts on finite quadratic modules and lattices, a proof of the structure theorem of finite quadratic modules is given. The proof is separated into five parts: Lemma \ref{lemm:DabcIsoToBorC}, Lemma \ref{lema:PriDecomOfFQM}, Lemma \ref{lemm:pGroupDecom}, Lemma \ref{lemm:2GroupDecom}, and Theorem \ref{thm:JordanDecompositionOfFQM}. We make some comments on this proof. Lemma \ref{lema:PriDecomOfFQM} deals with the primary decomposition of a finite quadratic module (which is unique). The proof of this lemma is according to \cite[Proposition 1.6]{Boylan2015}, where Boylan considers finite quadratic modules over number fields. Lemma \ref{lemm:pGroupDecom} and Lemma \ref{lemm:2GroupDecom} deal with decompositions of $p$-primary modules. The former assumes $p > 2$, and the latter $p = 2$. The proofs of these two lemmas are original. There is a technically difficult point in the proof of Lemma \ref{lemm:2GroupDecom}, for which we state a new lemma,  Lemma \ref{lemm:DabcIsoToBorC}, to solve this difficulty. Finally, we combine these lemmas, to obtain Theorem \ref{thm:JordanDecompositionOfFQM}. In Section \ref{sec:Discriminant modules of even lattices}, we turn to our second result---constructions of nondegenerate even lattices with given indecomposable finite quadratic modules. The resulting lattices are given by their Gram matrices in Theorem \ref{thm:LatticeOfLeastRankOfFQMApr}, Theorem \ref{thm:LatticeOfLeastRankOfFQMA2r}, and Theorem \ref{thm:LatticeOfLeastRankOfFQMBrCr}. Section \ref{sec:Discriminant modules of positive definite even lattices} has a similar structure, comparing to Section \ref{sec:Discriminant modules of even lattices}, in which we construct positive definite even lattices with given indecomposable finite quadratic modules as their discriminant modules.  The resulting lattices are given by their Gram matrices in Theorem \ref{thm:theLeastRankPosDefLatticeForAprEven}, Theorem \ref{thm:theLeastRankPosDefLatticeForAprOdd}, Theorem \ref{thm:theLeastRankPosDefLatticeForA2rEven}, Theorem \ref{thm:theLeastRankPosDefLatticeForA2rOdd}, and Theorem \ref{thm:theLeastRankPosDefLatticeForB2rC2r}.



\section{Bilinear map modules and quadratic map modules}
\label{sec:Billinear map modules and quadratic map modules}
\begin{conv}
In this section, we fix a commutative ring $R \neq \{0\}$ with an identity $1 = 1_{R}$. 
\end{conv}

The constructions presented here are common to both finite quadratic modules and $\mathbb{Z}$-lattices.

Let $M$, $M_1$, $M_2$ and $N$ be $R$-modules. We use $\mathrm{Hom}_{R}(M; N)$ to denote the set of linear maps from $M$ to $N$, and
$\mathrm{Hom}_{R}(M_1, M_2; N)$ the set of bilinear maps from $M_1 \times M_2$ to $N$. They both become $R$-modules under ordinary addition of functions and scalar multiplication of functions by elements of $R$. One can verify straightforwardly that the map
\begin{align}
\label{eq:IsoInducedByBilinearMapLeft}
\mathrm{Hom}_{R}(M_1, M_2; N) &\longrightarrow \mathrm{Hom}_{R}(M_1; \mathrm{Hom}_{R}(M_2; N)) \notag \\
B                &\longmapsto (m_1 \mapsto (m_2 \mapsto B(m_1, m_2))
\end{align}
is an isomorphism of $R$-modules. Similarly the map
\begin{align}
\label{eq:IsoInducedByBilinearMapRight}
\mathrm{Hom}_{R}(M_1, M_2; N) &\longrightarrow \mathrm{Hom}_{R}(M_2; \mathrm{Hom}_{R}(M_1; N)) \notag \\
B                &\longmapsto (m_2 \mapsto (m_1 \mapsto B(m_1, m_2))
\end{align}
is also an isomorphism. Some authors say that $B \in \mathrm{Hom}_{R}(M_1, M_2; N)$ is nondegenerate (or nonsingular) if both the values of the map \eqref{eq:IsoInducedByBilinearMapLeft} and \eqref{eq:IsoInducedByBilinearMapRight} at $B$ are $R$-linear isomorphisms from $M_1$ to $\mathrm{Hom}_R(M_2; N)$ and from $M_2$ to $\mathrm{Hom}_R(M_1; N)$ respectively. In some situations the term ``nondegenerate'' refers to a different meaning; it means that both the values of the map \eqref{eq:IsoInducedByBilinearMapLeft} and \eqref{eq:IsoInducedByBilinearMapRight} at $B$ are monomorphisms. To distinguish the two meanings we shall call $B$ nondegenerate (strongly nondegenerate respectively) if the values at $B$ are monomorphims (isomorphisms respectively) in this paper.

We are particularly interested in the case $M_1 = M_2 = M$. In this case we write $\underline{M}=(M, B)$ and call it a bilinear map module. If $N = R$ then $B$ is called a bilinear form and $\underline{M}=(M, B)$ a bilinear form module. We say that $\underline{M}$ is nondegenerate (strongly nondegenerate respectively) if B is nondegenerate (strongly nondegenerate respectively). \label{deff:StrongNondeg}

Among all the bilinear maps there are three important classes: the symmetric ones, the skew-symmetric (or antisymmetric) ones, and the alternating (or symplectic) ones. Assume that $M_1 = M_2 = M$. We say $B$ is symmetric (or $\underline{M}$ is a symmetric bilinear map module) if $B(x, y) = B(y, x)$ for any $x, y \in M$. We say that $B$ is skew-symmetric if $B(x, y) = -B(y, x)$ for any $x, y \in M$ and that $B$ is alternating if $B(x, x) = 0$ for any $x \in M$. We concern mainly the case that $B$ is symmetric.

There is an associated orthogonality relation on any bilinear map module $\underline{M}=(M, B)$. For $m_1, m_2 \in M$, we say that $m_1$ is orthogonal to $m_2$ (or $m_1 \perp m_2$ symbolically) if $B(m_1, m_2)=0$. One can verify that if $B$ is symmetric, skew-symmetric, or alternating, then the orthogonality relation is a symmetric relation, that is, $m_1 \perp m_2$ implies $m_2 \perp m_1$. When $R$ is an integral domain and $N=R$ one can also prove the converse: if orthogonality is a symmetric relation, then $B$ is symmetric or alternating, but this requires some effort. A proof can be found in \cite[p.~266]{Roman2011} for $R$ being a field. However, that proof is also valid for any integral domain $R$.

From now on, we always assume that $B$ is symmetric, skew-symmetric, or alternating, so that orthogonality is a symmetric relation. For a subset $S$ of $M$, denote by $S^{\perp}$ the set of all elements in $M$ which is orthogonal to elements in $S$. The subset $S^{\perp}$ must be a submodule. Here is a useful and basic property concerning orthogonality, which we omit the direct proof.

\begin{lemm}
\label{lemm:OrthoBasicProp}
Let $M$, $N$ be $R$-modules and $M_0$ a submodule of $M$. Let $B \in \mathrm{Hom}_{R}(M, M; N)$ be symmetric, skew-symmetric, or alternating. Denote by $\phi \colon M \rightarrow \mathrm{Hom}_R(M_0; N)$ the composite of the value of the map \eqref{eq:IsoInducedByBilinearMapLeft} at $B$ and the canonical map $\mathrm{Hom}_R(M; N) \rightarrow \mathrm{Hom}_R(M_0; N)$ that maps $f$ to the restriction of $f$ to $M_0$. Then $M_0^{\perp} = \ker(\phi)$. Consequently, if $\phi$ is surjection, then we have $M / M_0^{\perp} \cong \mathrm{Hom}_R(M_0; N)$.
\end{lemm}

Note that Lemma \ref{lemm:OrthoBasicProp} remains true if we substitute the map \eqref{eq:IsoInducedByBilinearMapRight} for \eqref{eq:IsoInducedByBilinearMapLeft}, since $B$ is symmetric, skew-symmetric, or alternating. Also note that we need $\phi$ to be surjection for the isomorphism relation holding. Hence, we give a useful condition for $\phi$ being surjection as follows:

\begin{lemm}
\label{lemm:PhiBeingSurjection}
Use notations and conditions of Lemma \ref{lemm:OrthoBasicProp}. Let $M_1$ be a submodule of $M$ containing $M_0$. If the restriction of $B$ to $M_1 \times M_1$ is strongly nondegenerate, and each map in $\mathrm{Hom}_R(M_0; N)$ extends to a map in $\mathrm{Hom}_R(M_1; N)$, then $\phi$ is a surjection.
\end{lemm}

The proof is also direct, so we omit it. To use this lemma in practice, the key point is to choose a suitable $M_1$ satisfying the two conditions. One way of doing this in the case $N=R$ is to use Lemma \ref{lemm:GramMatrixNeqR}.

The following lemma is well-known and useful, so we give a detailed proof. (See, for instance, Milnor and Husemoller \cite{MH1973} and Wall \cite{Wall1963}.)

\begin{lemm}
\label{lemm:BasicLemmaOfStrongliNondeg}
Use notations and conditions of Lemma \ref{lemm:OrthoBasicProp}. Suppose that the restriction of $B$ to $M_0 \times M_0$ is strongly nondegenerate. Then we have
\begin{enumerate}
\item $M=M_0 \oplus M_0^{\perp}$.
\item If $B$ is also strongly nondegenerate, then so is the restriction of $B$ to $M_0^{\perp} \times M_0^{\perp}$.
\end{enumerate}
\end{lemm}

\begin{proof}
We prove the first part. Need to prove that $M_0 \cap M_0^{\perp} = \{0\}$ and $M = M_0 + M_0^{\perp}$. To prove the former assertion, let $x \in M_0 \cap M_0^{\perp}$. Then for any $y \in M_0$ we have $B(x, y) = 0$. It follows that $x = 0$ since $x \in M_0$ and the restriction of $B$ to $M_0 \times M_0$ is strongly nondegenerate. Next we prove the latter assertion. Let $x \in M$ be arbitary. Consider the map $y \mapsto B(x, y)$ on $M_0$. It is a linear map on $M_0$, as easily verified. Hence by strong nondegeneracy of the restriction of $B$ to $M_0 \times M_0$, there exists a unique $x_0 \in M_0$ such that $B(x, y) = B(x_0, y)$ for $y \in M_0$. It follows that $x - x_0 \in M_0^{\perp}$, and $x = x_0 + (x - x_0) \in M_0 + M_0^{\perp}$. This concludes the proof of the first part.

We now prove the second part. It suffices to prove that the map $\pi \colon M_0^{\perp} \rightarrow \mathrm{Hom}_R(M_0^{\perp}; N)$, $x \mapsto (y \mapsto B(x, y))$ is an isomorphism of $R$-modules. By the first part of this lemma, any map in $\mathrm{Hom}_R(M_0^{\perp};N)$ extends to a map in $\mathrm{Hom}_R(M;N)$. Since $B$ is strongly nondegenerate, for any $f \in \mathrm{Hom}_R(M_0^{\perp}; N)$, there exists an $x \in M$ such that $B(x, y) = f(y)$ for $y \in M_0^{\perp}$. We write $x = x_1 + x_2$ with $x_1 \in M_0$ and $x_2 \in M_0^{\perp}$ by the first part of this lemma. One then verifies immediately thay $\pi(x_2) = f$, hence $\pi$ is surjective. Since $\pi$ is obviously $R$-linear, it remains to prove that $\pi$ is injective, i.e., that $\ker(\pi) = 0$. Let $x \in \ker(\pi)$, that is, $B(x, y) = 0$ for any $y \in M_0^{\perp}$. It follows that $x \perp M$. Consequently $x = 0$ by strong nondegeneracy of $B$. This concludes the proof of the second part.
\end{proof}

We now turn to the subject of quadratic maps. Let $M, N$ be $R$-modules and let $Q \colon M \rightarrow N$. We call $Q$ a quadratic map if the following conditions hold:
\begin{enumerate}
\item $Q(rx) = r^2Q(x)$ for $r\in R$, $x \in M$.
\item The associated map $B_Q(x, y) := Q(x+y)-Q(x)-Q(y)$ is bilinear.
\end{enumerate}
Define $\underline{M}=(M,Q)$ and call it a quadratic map module or quadratic module (over the ring $R$). The bilinear map module $(M, B_Q)$ is called the associated bilinear map module of $\underline{M}$. Note that the bilinear map $B_Q$ associated with $Q$ is necessarily symmetric. In fact, when the module $N$ satisfies that for any $v \in N$, there exists a unique $w \in N$ such that $v = w + w$ (denote w by $2^{-1}v$ even if 2 is not a unit in $R$) there is a one-to-one correspondence between quadratic maps and symmetric bilinear maps. The correspondence is given by the following two maps: the map $Q \mapsto B_Q$ from quadratic maps onto symmetric bilinear maps, and the map $B \mapsto (x \mapsto 2^{-1}B(x, x))$ from symmetric bilinear maps onto quadratic maps. They are inverse to each other, as straightforwardly verified. Thus, under this circumstance, the theory of symmetric bilinear maps and quadratic maps are indistinguishable. However, when $N$ does not satisfy that condition, the two theories are different. This happens, for instance, when $R$ is a filed of characteristic 2 and $N \neq \{0\}$, or when $N$ has torsion elements of order $2$. The relation of $Q$ and $B_Q$ can be easily generalized to the case of more than two summands by induction:
\begin{equation}
\label{eq:RelationBetQandBQ}
Q\left(\sum_{1 \leq i \leq n}x_i\right)=\sum_{1 \leq i \leq n}Q(x_i) + \sum_{1 \leq i < j \leq n}B_Q(x_i,x_j).
\end{equation}

We can form orthogonal sums, quotients, and tensor products of bilinear map modules or quadratic modules. The construction of tensor products need $N$ to be an associative algebra over $R$ and plays an important role in the theory of bilinear form module: It is used to define the multiplication on the Witt ring of $R$. We shall not need this construction. For details, see Milnor and Husemoller \cite[p.~14]{MH1973}. Nevertheless, the constructions of orthogonal sums and quotients are important in subsequent sections, so we review these constructions here.

Let $\underline{M}_i=(M_i, B_i \colon M_i \times M_i \rightarrow N)$, $i \in I$ be a family of bilinear map modules. Recall that we always assume that $B$ is symmetric, skew-symmetric, or alternating. Define its orthogonal (external\footnote{Similarly one can consider orthogonal internal direct sum which is isometrically isomorphic to orthogonal external direct sum. By abuse of language, we do not distinguish these two concepts.}) direct sum as $\bigoplus_{i \in I}M_i$ equipped with the bilinear map defined as follows:
\begin{equation*}
B\left((x_i)_{i \in I}, (y_i)_{i \in I} \right) = \sum_{i \in I}B_i(x_i, y_i).
\end{equation*}
Note that external direct sum $\bigoplus_{i \in I}M_i$ is the set of elements in Cartesian product $\prod_{i \in I}M_i$ with a finite support; hence the sum in the above formula makes sense. It can be deduced that if all $B_i$'s are nondegenerate, then $B$ is nondegenerate. Moreover, assume that $I$ is finite. Then the condition that all $B_i$'s are strongly nondegenerate implies that $B$ is strongly nondegenerate. In addition, if all $B_i$ are symmetric (skew-symmetric, alternating respectively), then $B$ is symmetric (skew-symmetric, alternating respectively).

The construction of orthogonal sums also applies to quadratic modules. Let $\underline{M}_i=(M_i, Q_i \colon M_i \rightarrow N)$, $i \in I$ be a family of quadratic modules. Then define its orthogonal sum as $\bigoplus_{i \in I}M_i$ equipped with the quadratic map defined as follows:
\begin{equation*}
Q\left((x_i)_{i \in I}\right) = \sum_{i \in I}Q_i(x_i).
\end{equation*}
It follows that $B_Q\left((x_i)_{i \in I}, (y_i)_{i \in I} \right) = \sum_{i \in I}B_{Q,i}(x_i, y_i)$, where $B_Q$ is the associated bilinear map of $Q$, and $B_{Q, i}$ the associated bilinear map of $Q_i$. Consequenty, if all the summands are nondegenerate quadratic modules, then their orthogonal sum is nondegenerate. When $I$ is finite this fact remains true if we substitute strong nondegeneracy for nondegeneracy.

\label{deff:Quotients}We also need the concept of quotients of bilinear map modules or quadratic modules. Let $M$, $N$ be $R$-modules and let $B \in \mathrm{Hom}_R(M, M; N)$. We write $\underline{M}=(M, B)$. Let $M_0$ and $N_0$ be submodules of $M$ and $N$ respectively. Suppose that $B(x, y) \in N_0$ for any $(x,y) \in M\times M_0$ or $M_0 \times M$. Then we define the quotient of $\underline{M}$ with respect to $(M_0, N_0)$ as the $R$-module $M/M_0$ equipped with the bilinear map
\begin{align*}
M/M_0 \times M/M_0 &\longrightarrow N/N_0 \\
(x + M_0, y + M_0) &\longmapsto B(x, y) + N_0.
\end{align*}
One can immediately verify that this map is well-defined and is indeed $R$-bilinear, hence gives $M/M_0$ a bilinear map module structure. In addition, if $B$ is symmetric (skew-symmetric, or alternating respectively), then the map on quotient module is alse symmetric (skew-symmetric, or alternating respectively).

In the quadratic module case, $B$ is replaced by a quadratic map $Q\colon M\rightarrow N$ and $\underline{M}=(M,Q)$. The condition for $B$ is replaced by a condition for $Q$: $Q(x) \in N_0$ for any $x \in M_0$ and $B_Q(x, y) \in N_0$ for any $(x, y) \in M \times M_0$. Then we can define the quotient of $\underline{M}$ with respect to $(M_0, N_0)$ the $R$-module $M/M_0$ equipped with the quadratic map
\begin{align*}
M/M_0 &\longrightarrow N/N_0 \\
x + M_0 &\longmapsto Q(x) + N_0
\end{align*}
which gives $M/M_0$ a quadratic module structure.

Like all other mathematical structures, there are morphisms between bilinear map modules which make them a category. A morphism from a bilinear map module $\underline{M}_1=(M_1, B_1 \colon M_1 \times M_1 \rightarrow N)$ to another $\underline{M}_2=(M_2, B_2 \colon M_2 \times M_2 \rightarrow N)$ is a $R$-linear map $\pi$ from $M_1$ into $M_2$ preserving bilinear maps, i.e. satisfying $B_1(x, y)=B_2(\pi(x),\pi(y))$. If, in addition, $\pi$ is a bijection, we call $\pi$ an isometric isomorphism. For quadratic modules, we can define morphisms and isometric isomorphisms similarly: just replace the condition $B_1(x, y)=B_2(\pi(x),\pi(y))$ with $Q_1(x)=Q_2(\pi(x))$.

We conclude this section by introducing methods of matrix theory into the theory of bilinear map modules. To this end, we need $N$ to be an unitary associative $R$-algebra, that is, $N$ is both an unitary ring and $R$-module such that the ring multiplication are $R$-bilinear. Thus we have a category of matrices whose objects are positive integers and morphisms from $n$ to $m$ are matrices of $m$ rows and $n$ columns with entries in $N$. The composition of this category is the usual matrix multiplication. A direct verification (take care that we do not assume that $N$ is commutative) shows these data indeed make up a category. In another words, the matrix multiplication over the ring $N$ is associative and for each positive integer $n$ there is a unique $n \times n$ identity matrix even if $N$ is non-commutative.

Recall that we let $M_1$ and $M_2$ be $R$-modules and $B \in \mathrm{Hom}_R(M_1, M_2; N)$. Let $\mathbf{x} = (x_1, x_2, \ldots, x_s)$ and $\mathbf{y} = (y_1, y_2, \ldots, y_t)$ be finite sequences in $M_1$ and $M_2$ respectively. Define the Gram matrix of $B$ with respect to the pair $(\mathbf{x}, \mathbf{y})$ (or simply, with respect to $\mathbf{x}$ if $\mathbf{x}=\mathbf{y}$) to be the matrix$\left(B(x_i, y_j)\right)_{1 \leq i \leq s, \, 1 \leq j \leq t}$. Denote it by $G_B$ if the sequences $\mathbf{x}$ and $\mathbf{y}$ are clear. A fundamental method to calculate the value of a bilinear map is as follows. Let $\alpha_1, \alpha_2, \ldots, \alpha_s$ and $\beta_1, \beta_2, \ldots, \beta_t$ be in $R$ and denote the multiplication identity of $N$ by $1_N$. Then we have
\begin{equation}
\label{eq:CalculateBilinearMapValue}
B\left(\sum_{i=1}^s \alpha_i x_i, \sum_{j=1}^t \beta_j y_j\right) = \left(\alpha_1 1_N, \ldots, \alpha_s 1_N\right) \cdot G_B \cdot \left(\beta_1 1_N, \ldots, \beta_t 1_N\right)^{\mathrm{T}}.
\end{equation}

Recall that if $A$ is a matrix with $m$ rows and $n$ columns, then by a left inverse of $A$ we mean a matrix $B$ with $n$ rows and $m$ columns such that $B \cdot A$ is the $n \times n$ identity matrix. Similarly we can define the notion of right inverses of $A$.

\begin{lemm}
\label{lemm:GramMatrixLeftRightInverse}
Assume that the map $R \rightarrow N, \, r \mapsto r\cdot1_N$ is injective. If $G_B$ has a right inverse, then $(x_1, x_2, \ldots, x_s)$ is $R$-linearly independent. Likewise, if $G_B$ has a left inverse, then $(y_1, y_2, \ldots, y_t)$ is $R$-linearly independent.
\end{lemm}

\begin{proof}
We only prove for the case that $G_B$ has a right inverse. Let  $\alpha_1, \alpha_2, \ldots, \alpha_s$ be in $R$ such that $\sum_{i=1}^s \alpha_i x_i=0$. It follows from \eqref{eq:CalculateBilinearMapValue} that
\begin{equation*}
\left(B\left(\sum_{i=1}^s \alpha_i x_i, y_1\right),\dots, B\left(\sum_{i=1}^s \alpha_i x_i, y_t\right)\right) = \left(\alpha_1 1_N, \ldots, \alpha_s 1_N\right) \cdot G_B \cdot \mathrm{I}_t
\end{equation*}
where $\mathrm{I}_t$ is the $t \times t$ identity matrix. Multiplying a right inverse of $G_B$ on both sides and using the fact that $\sum_{i=1}^s \alpha_i x_i=0$ we obtain that $\alpha_1 1_N= \ldots= \alpha_s 1_N = 0_N$. Hence $\alpha_1= \ldots= \alpha_s = 0_N$ since the map $r \mapsto r\cdot1_N$ is injective. This proves that $(x_1, x_2, \ldots, x_s)$ is linearly independent.
\end{proof}

In the case $M_1 = M_2 = M$ and $N=R$ one can prove more.
\begin{lemm}
\label{lemm:GramMatrixNeqR}
Let $M_1 = M_2 = M$ and $N=R$. Let Let $\mathbf{x} = (x_1, x_2, \ldots, x_s)$ be a sequence in $M$ and $G_B$ the Gram matrix with respect to $\mathbf{x}$. Define $M^\prime = \sum_{1 \leq i \leq s}Rx_i$. Then $G_B$ is invertible if and only if the sum defining $M^\prime$ is a direct sum and the restriction of $B$ to $M^\prime \times M^\prime$ is strongly nondegenerate.
\end{lemm}

\section{Lattices and finite quadratic modules}
\label{sec:Lattices and finite quadratic modules}
In this section, we recall some basic definitions and properties on lattices over the integers, and on finite quadratic modules. Then we give a complete proof of the Jordan decomposition theorem of finite quadratic modules.
\begin{deff}
\label{deff:DefOfLattices}
Let $L$ be a $\mathbb{Z}$-module, and $B$ a symmetric bilinear form from $L \times L$ into the real numbers $\mathbb{R}$. Set $\underline{L}=(L, B)$. We say that $\underline{L}$ is a lattice (over the integers) if $L$ is free of finite rank (i.e. has a finite $\mathbb{Z}$-basis).
\end{deff}

The quadratic form $Q\colon L \rightarrow \mathbb{R},\, x \mapsto 2^{-1}B(x,x)$ is called the quadratic form of the lattice $\underline{L}$. If the image of $B$ is contained in $\mathbb{Z}$, then we call $\underline{L}$ integral; moreover, if the image of $Q$ is contained in $\mathbb{Z}$ we call $\underline{L}$ even integral (or simply call it even). It follows that $\underline{L}$ is even if and only if $B(x, x) \in 2\mathbb{Z}$ for any $x \in L$. We call an integral lattice which is not even an odd integral lattice, or simply an odd lattice. If for any nonzero $x \in L$ we have $Q(x) > 0$, then we call $\underline{L}$ a positive definite lattice.

\begin{rema}
\label{rema:RemarkOnDefOfLattices}
If we have an ambient real sysmetric bilinear form vector space $\underline{V} = (V, B \colon V\times V \rightarrow \mathbb{R})$), then we say that $L$ is a $\mathbb{Z}$-lattice in $\underline{V}$ in a more strict sense: It means not only that $L$ is a free $Z$-module of rank $\dim_{\mathbb{R}}V$ contained in $V$, but also that $L$ has a $\mathbb{Z}$-basis which is also a $\mathbb{R}$-basis of $V$ additionally.
\end{rema}

Let $R$ be a commutative extension ring of $\mathbb{Z}$, we shall use the tensor product $R \otimes_\mathbb{Z} L$ which is a free $R$-module. Note that the bilinear form $B$ extends to a bilinear form on $R \otimes_\mathbb{Z} L$. We are interested mainly in the case that $R$ is the field of rationals $\mathbb{Q}$, of reals $\mathbb{R}$, or of complex numbers $\mathbb{C}$. We choose realizations of tensor products such that $L \subseteq \mathbb{Q} \otimes L \subseteq \mathbb{R} \otimes L \subseteq \mathbb{C} \otimes L$.

\begin{conv}
We are only interested in nondegeneracy case. In the followings, we always assume that the bilinear extension $B^\prime$ of $B$ to $\mathbb{R} \otimes L$ is strongly nondegenerate.
Note that $B^\prime$ is strongly nondegenerate if and only if it is nondegenerate, since $\mathbb{R}$ is a field and $\dim\mathbb{R}\otimes L < \infty$. The following example shows that saying that $B$ is nondegenerate is not enough.
\end{conv}

\begin{examp}
Let $L=\mathbb{Z}^2$ and $(e_1,e_2)$ be the standard basis. Define a bilinear form $B$ as $B(e_1,e_1)=1,\,B(e_1,e_2)=B(e_2,e_1)=\sqrt{2},\,B(e_2,e_2)=2$. Then the map $L \rightarrow \mathrm{Hom}_{\mathbb{Z}}(L;\mathbb{R}),\,v \mapsto (w \mapsto B(v,w))$ is a $\mathbb{Z}$-monomorphism, thus $B$ is nondegenerate in the sense in Section \ref{sec:Billinear map modules and quadratic map modules}. However, since the Gram matrix with respect to this basis is not invertible even over the field $\mathbb{R}$, we have that the bilinear extension of $B$ to $\mathbb{R} \otimes L$ is not nondegenerate.
\end{examp}

We define the dual lattice of $\underline{L}$ which is important in the theory of lattices as follows.

\begin{deff}
\label{deff:DefOfDualLattice}
The dual lattice of $\underline{L}$ is the lattice contained in the $\mathbb{R}$-space $\mathbb{R} \otimes L$ whose element $v$ satisfies $B(v, w) \in \mathbb{Z}$ for any $w \in L$. The dual lattice is denoted by $L^\sharp$ which inherits the bilinear form $B$ defined on $\mathbb{R} \otimes L$. When $L$ is a lattice in an ambient vector space $\underline{V} = (V, B \colon V\times V \rightarrow \mathbb{R})$ with a nondegenerate sysmetric bilinear form $B$, we can also define the dual lattice of $L$ in $\underline{V}$ in a similar manner.
\end{deff}

Let $(e_1,\dots,e_n)$ be a $\mathbb{Z}$-basis of $\underline{L}$, namely, $L=\bigoplus_{1 \leq i \leq n}\mathbb{Z}e_i$. Let $G_B$ be the Gram matrix with respect to this basis. Then $\mathop{\mathrm{det}}G_B \neq 0$. Thus, for each $\mathbb{R}$-linear form $\mathbb{R} \otimes L \rightarrow \mathbb{R},\,e_j \mapsto \delta_{ij}$, we can find an $e_i^{\ast} \in \mathbb{R} \otimes L$ such that $B(e_i^{\ast},e_j) = \delta_{ij}$. One verifies immediately that $L^\sharp=\bigoplus_i \mathbb{Z}e_i^\ast$. The basis $(e_1^\ast, \dots, e_n^\ast)$ of $L^\sharp$ is called the dual basis of $(e_1,\dots,e_n)$ with respect to $B$ in $\mathbb{R} \otimes L$.

\begin{prop}
\label{propBasisPropOnDualLattice}
Let $\underline{L}=(L,B)$ be a lattice of rank $n \in \mathbb{Z}_{\geq 1}$ and let $(e_1,\dots,e_n)$ be a $\mathbb{Z}$-basis of $\underline{L}$. Let $G_B$ be the Gram matrix of $B$ with respect to this basis. Then
\begin{enumerate}
\item The dual basis can be expressed as $(e_1^\ast, \dots, e_n^\ast)^{\mathrm{T}}$ = $G_{B}^{-1}\cdot (e_1,\dots,e_n)^{\mathrm{T}}$.
\item The Gram matrix of $\underline{L}^\sharp$ with respect to $(e_1^\ast, \dots, e_n^\ast)$ is $G_B^{-1}$.
\item The double dual lattice $(L^\sharp)^\sharp$ as a lattice in $\mathbb{R} \otimes L$ is equal to $L$.
\item The lattice $\underline{L}$ is integral if and only if $L \subseteq L^\sharp$, if and only if entries of $G_B$ are all integers.
\end{enumerate}
\end{prop}

\begin{proof}
Write $(e_1^\ast, \dots, e_n^\ast)^{\mathrm{T}}$ = $A\cdot (e_1,\dots,e_n)^{\mathrm{T}}$, where $A=(a_{ij})$ is an $n \times n$ matrix. It follows from the fact $B(e_i^{\ast},e_j) = \delta_{ij}$ and \eqref{eq:CalculateBilinearMapValue} that $\mathrm{I}_n = AG_B\mathrm{I}_n$. The first part of the theorem then follows. The second part follows from the first part, \eqref{eq:CalculateBilinearMapValue} and the fact that $G_B^{-1}$ is symmetric. To prove the third part, note that $e_i^{\ast\ast}=e_i$. The last part of the theorem is obvious by definition.
\end{proof}

Let $\underline{L}=(L,B)$ be an integral lattice. Then by Proposition \ref{propBasisPropOnDualLattice} we have $L \subseteq L^\sharp$. We now use the construction of quotients in Section \ref{deff:Quotients}. Since $B(L, L^\sharp)=B(L^\sharp,L) \subseteq \mathbb{Z}$ and $B(L^\sharp,L^\sharp) \subseteq \mathbb{Q}$, the map $L^\sharp/L \times L^\sharp/L \rightarrow \mathbb{Q}/\mathbb{Z}, \,(v+L, w+L) \mapsto B(v,w)+\mathbb{Z}$ is well-defined and $\mathbb{Z}$-bilinear where $\mathbb{Q}/\mathbb{Z}$ is regarded as a quotient of $\mathbb{Z}$-modules and hence is itself a $\mathbb{Z}$-module. It can be verified that the bilinear map module $(L^\sharp/L, (v+L,w+L) \mapsto B(v,w)+\mathbb{Z})$ is nondegenerate by Proposition \ref{propBasisPropOnDualLattice}(3).

If, moreover, $\underline{L}=(L,B)$ is an even integral lattice, we can regard $\underline{L}$ as a quadratic map module (In this situation, the theory of quadratic maps and symmetric bilinear maps are indistinguishable, as explained in Section \ref{sec:Billinear map modules and quadratic map modules}.). Now we have $Q(L) \subseteq \mathbb{Z}$ and $B(L, L^\sharp) \subseteq \mathbb{Z}$. Thus the map $L^\sharp/L \rightarrow \mathbb{Q}/\mathbb{Z},\, v+L \mapsto Q(v)+\mathbb{Z}$ is well-defined and quadratic, as explained in Section \ref{sec:Billinear map modules and quadratic map modules}. This defines a quadratic module $(L^\sharp/L,\, v+L \mapsto Q(v)+\mathbb{Z})$ whose corresponding bilinear map is $L^\sharp/L \times L^\sharp/L \rightarrow \mathbb{Q}/\mathbb{Z}, \,(v+L, w+L) \mapsto B(v,w)+\mathbb{Z}$. It follows that the quadratic module just defined is nondegenerate. We will give a proof that it is actually strongly nondegenerate later! See Corollary \ref{coro:FQMStroNondeg}.

No matter $L$ is even or old, the quotient $L^\sharp/L$ must be a finite abelian group since elements in $L^\sharp/L$ are all torsion elements and $L^\sharp/L$ is finitely generated. A well-known argument shows that $\vert L^\sharp/L\vert = \vert\mathop{\mathrm{det}}G_B\vert$ where $G_B$ is the Gram matrix with respect to any $\mathbb{Z}$-basis of $L$.

The finite abelian group (i.e. finite $\mathbb{Z}$-module) $L^\sharp/L$ together with the quadratic map $v+L \mapsto Q(v)+\mathbb{Z}$ where $L$ is even is a primary example of the following concept:

\begin{deff}
\label{deff:DefOfFQM}
Let $M$ be a finite abellian group, and $Q:M\rightarrow \mathbb{Q}/\mathbb{Z}$ be a quadratic map. If the corresponding bilinear map $(x,y) \mapsto Q(x+y)-Q(x)-Q(y)$ is nondegenerate in the sense defined in Section \ref{sec:Billinear map modules and quadratic map modules}, then the pair $\underline{M}=(M,Q)$ is called a finite quadratic module.
\end{deff}

Note that comparing to the general definition of quadratic map modules in Section \ref{sec:Billinear map modules and quadratic map modules}, here we add an additional condition of nondegeneracy. In the followings, when we speak of a finite quadratic module, we always mean a nondegenerate one. As a convention, we also treat the zero module with the zero map as a finite quadratic module (in fact, this convention coincides with the concepts in Section \ref{sec:Billinear map modules and quadratic map modules}).

\begin{examp}
\label{examp:DiscModule}
Let $\underline{L}=(L, B)$ be an even lattice whose quadratic form is $Q$, then $(L^\sharp/L,\, v+L \mapsto Q(v)+\mathbb{Z})$ is a finite quadratic module, as discussed before Definition \ref{deff:DefOfFQM}. This is called the discriminant module (group) of $\underline{L}$ and is denoted by $D_{\underline{L}}$. Note that the discriminant module of the zero lattice is the zero module.
\end{examp}

The constructions discussed in Section \ref{sec:Billinear map modules and quadratic map modules}, such as orthogonal sums, isometric isomorphisms on general quadratic map modules or bilinear map modules carry over to finite quadratic modules.

The following lemma is useful and basic in the theory of finite quadratic modules.
\begin{lemm}
\label{lemm:QoverZIsoCmul}
Denote by $\mathbb{C}^\times$ the multiplicative group of nonzero complex numbers (regarding as a $\mathbb{Z}$-module). Then the map $\mathbb{Q}/\mathbb{Z} \rightarrow \mathbb{C}^\times,\,a+\mathbb{Z} \mapsto \mathrm{e}^{2\uppi \mathrm{i}a}$ is a $\mathbb{Z}$-module monomorphism whose image is the torsion elements in $\mathbb{C}^\times$. Moreover, let $M$ be a finite $\mathbb{Z}$-module and denote by $\widehat{M}$ the dual group of $M$, i.e., the group of all linear characters of $M$. Then $\mathrm{Hom}_{\mathbb{Z}}(M; \mathbb{Q}/\mathbb{Z})$ is isomorphic to $\widehat{M}$ as $\mathbb{Z}$-modules.
\end{lemm}
\begin{proof}
It is obviously that $a+\mathbb{Z} \mapsto \mathrm{e}^{2\uppi \mathrm{i}a}$ is monomorphism and the image is the torsion elements in $\mathbb{C}^\times$. It remains to prove that $\mathrm{Hom}_{\mathbb{Z}}(M; \mathbb{Q}/\mathbb{Z})$ is isomorphic to $\widehat{M}$. Define a map $\pi$ from $\mathrm{Hom}_{\mathbb{Z}}(M; \mathbb{Q}/\mathbb{Z})$ to $\widehat{M}$ as follows: for $f \in \mathrm{Hom}_{\mathbb{Z}}(M; \mathbb{Q}/\mathbb{Z})$, let $\pi(f)$ be the map whose value at $x \in M$ is $\mathrm{e}^{2 \uppi \mathrm{i} f(x)}$. One can verify that $\pi$ is well-defined and gives the desired isomorphism.
\end{proof}
\begin{coro}
\label{coro:FQMStroNondeg}
Let $\underline{M}=(M,Q)$ be a finite quadratic module (or more generally a finite bilinear map module with the bilinear map nondegenerate). Then $\underline{M}$ is strongly nondegenerate.
\end{coro}
\begin{proof}
By definition of strong nondegeneracy in Section \ref{sec:Billinear map modules and quadratic map modules}, we need to prove that the map $M \rightarrow \mathrm{Hom}_{\mathbb{Z}}(M; \mathbb{Q}/\mathbb{Z}),\,x \mapsto (y \mapsto B(x,y))$ is an isomorphism where $B$ is the bilinear map associated with $Q$ (i.e., $B(x,y)=Q(x+y)-Q(x)-Q(y)$). It is already a monomorphism since $B$ is nondegenerate by assumption. It is indeed an isomorphism since $\vert \mathrm{Hom}_{\mathbb{Z}}(M; \mathbb{Q}/\mathbb{Z}) \vert = \vert \widehat{M} \vert = \vert M \vert < \infty$ by Lemma \ref{lemm:QoverZIsoCmul} and elementary theory of finite abelian groups.
\end{proof}

The next is another important property of finite quadratic modules.
\begin{prop}
\label{prop:HomIsoQuoAndDoubleOrtho}
Let $\underline{M}=(M,Q)$ be a finite quadratic module and $M_0$ be a submodule. Then the map $\phi$ in Lemma \ref{lemm:OrthoBasicProp} (with $R=\mathbb{Z}$ and $N = \mathbb{Q}/\mathbb{Z}$) is actually a surjection. Hence $M / M_0^{\perp} \cong \mathrm{Hom}_\mathbb{Z}(M_0; \mathbb{Q}/\mathbb{Z})$ (the canonical map is an isomorphism). Moreover, we have $\vert M_0\vert \cdot \vert M_0^\perp \vert = \vert M \vert$ and $(M_0^\perp)^\perp = M_0$.
\end{prop}
A proof can be found in \cite[Proposition 1.7]{Boylan2015}. For the reader's convenience, we also give a proof.
\begin{proof}
Let $B$ be the associated bilinear map (i.e., $B(x,y)=Q(x+y)-Q(x)-Q(y)$). Firstly, we prove that the map $\phi$ is a surjection. By Lemma \ref{lemm:PhiBeingSurjection} it suffices to prove that $B$ is strongly nondenegerate, and each map in $\mathrm{Hom}_{\mathbb{Z}}(M_0; \mathbb{Q}/\mathbb{Z})$  extends to a map in $\mathrm{Hom}_{\mathbb{Z}}(M; \mathbb{Q}/\mathbb{Z})$. Corollary \ref{coro:FQMStroNondeg} implies immediately the former assertion. To prove the latter, note that it is equivalent to saying that each linear character of $M_0$ extends to a linear character of $M$ by Lemma \ref{lemm:QoverZIsoCmul}. This is well-known. See, for example, \cite[Chap. VI, Sect. 1, Proposition 1]{Serre1973}. Secondly, it follows from Lemma \ref{lemm:OrthoBasicProp} and the fact just proved immediately that $M / M_0^{\perp} \cong \mathrm{Hom}_\mathbb{Z}(M_0; \mathbb{Q}/\mathbb{Z})$. Finally, we prove that $\vert M_0\vert \cdot \vert M_0^\perp \vert = \vert M \vert$ and $(M_0^\perp)^\perp = M_0$. By the second part of this proposition and Lemma \ref{lemm:QoverZIsoCmul} we have $\vert M / M_0^{\perp} \vert = \vert \mathrm{Hom}_\mathbb{Z}(M_0; \mathbb{Q}/\mathbb{Z}) \vert = \vert M_0 \vert$. Hence $\vert M_0\vert \cdot \vert M_0^\perp \vert = \vert M \vert$. Substitute $M_0^\perp$ for $M_0$; we obtain $\vert M_0^\perp \vert \cdot \vert (M_0^\perp)^\perp \vert = \vert M \vert$, and hence $ \vert M_0 \vert = \vert (M_0^\perp)^\perp \vert$. It follows immediately that $M_0 = (M_0^\perp)^\perp$ since  $M_0 \subseteq (M_0^\perp)^\perp$ and $M_0$ is finite.
\end{proof}

We now turn to give concrete examples of finite quadratic modules. The modules given in the following definition turn out to be cornerstones of all finite quadratic modules---any finite quadratic module can be decomposed as an orthogonal sum of these modules (up to isomorphism).

\begin{deff}
\label{deff:FQMABC}
Let $p$ be a prime, $r$ a positive integer, $a$ an integer not divisible by $p$. Define the following finite quadratic modules:
\begin{align}
\label{eq:DefAp}\underline{A}_{p^r}^a &:= \left(\mathbb{Z}/p^r\mathbb{Z},\,x+p^r\mathbb{Z} \mapsto \frac{a}{p^r}x^2+\mathbb{Z}\right) \, \text{for }p>2, \\
\label{eq:DefA2}\underline{A}_{2^r}^a &:= \left(\mathbb{Z}/2^r\mathbb{Z},\,x+2^r\mathbb{Z} \mapsto \frac{a}{2^{r+1}}x^2+\mathbb{Z}\right) \, \text{for } p=2, \\
\label{eq:DefB2}\underline{B}_{2^r}   &:= \left(\mathbb{Z}/2^r\mathbb{Z} \times \mathbb{Z}/2^r\mathbb{Z},\,(x+2^r\mathbb{Z},y+2^r\mathbb{Z}) \mapsto \frac{x^2+xy+y^2}{2^r}+\mathbb{Z} \right), \\
\label{eq:DefC2}\underline{C}_{2^r}   &:= \left(\mathbb{Z}/2^r\mathbb{Z} \times \mathbb{Z}/2^r\mathbb{Z},\,(x+2^r\mathbb{Z},y+2^r\mathbb{Z}) \mapsto \frac{xy}{2^r}+\mathbb{Z}\right).
\end{align}
\end{deff}
\begin{rema}
\label{rema:BasicPropOfSimpleFQM}
One can verify that the quadratic maps of these modules are well-defined and the associated bilinear maps are nondegenerate. Thus they are indeed finite quadratic modules. Moreover, the modules $\underline{A}_{p^r}^a$ are indecomposable  as abelian groups since they are primary cyclic. On the other hand, although the modules $\underline{B}_{2^r}$ and $\underline{C}_{2^r}$ are decomposable as abelian groups, they can not be written as orthogonal sums of proper submodules.
\end{rema}

\begin{rema}
\label{rema:SimpleModIso}
It can be verified that modules of different types described in Definition \ref{deff:FQMABC} are not isomorphic. Maybe the only case which needs explaining is that $\underline{B}_{2^r} \not\cong \underline{C}_{2^r}$. This can be verified by observing that there are exactly $4^r - 4^{r-1}$ elements in $\underline{B}_{2^r}$ such that the value of the quadratic map is an odd number divided by $2^r$ mod $\mathbb{Z}$, while there are only $4^{r-1}$ elements in $\underline{C}_{2^r}$ satisfying this condition. Now consider modules of the same type. For an odd prime $p$ and $a, b$ not divisible by $p$,  we can prove that $\underline{A}_{p^r}^a$ is isomorphic to $\underline{A}_{p^r}^b$ if and only if the Legendre symbol $\legendre{a}{p}=\legendre{b}{p}$. Thus, for a fixed prime $p>2$ and fixed $r \in \mathbb{Z}_{\geq 1}$, the modules $\underline{A}_{p^r}^a$ give exactly two isomorphism classes when $a$ runs through the integers not divisible by $p$. Now for $a,b$ not divisible by $2$, we have $\underline{A}_{2^r}^a$ is isomorphic to $\underline{A}_{2^r}^b$ if and only if $ab^{-1}$ is a quadratic residue modulo $2^{r+1}$. Thus, for $p=2$ and $r=1$, the modules $\underline{A}_{2}^a$ give exactly two isomorphism classes, and for $p=2$ and $r>1$, the modules $\underline{A}_{2^r}^a$ give exactly four isomorphism classes when $a$ runs through odd integers.
\end{rema}

\begin{examp}
Besides modules in Definition \ref{deff:FQMABC}, there are still other modules which can be defined in a natural way, although they may be decomposable, or isomorphic to a module given in Definition \ref{deff:FQMABC}. We define
\begin{align}
\label{eq:DefAm}\underline{A}_{m} &:= \left(\mathbb{Z}/m\mathbb{Z},\,x+m\mathbb{Z} \mapsto \frac{1}{m}x^2+\mathbb{Z}\right) \, \text{for }2 \nmid m, \\
\label{eq:DefAm2}\underline{A}_{m} &:= \left(\mathbb{Z}/m\mathbb{Z},\,x+m\mathbb{Z} \mapsto \frac{1}{2m}x^2+\mathbb{Z}\right) \, \text{for }2 \mid m, \\
\label{eq:DefD}\underline{D}_{2^r}^{a,b,c}   &:= \left(\mathbb{Z}/2^r\mathbb{Z} \times \mathbb{Z}/2^r\mathbb{Z},\,(x+2^r\mathbb{Z},y+2^r\mathbb{Z}) \mapsto \frac{ax^2+bxy+cy^2}{2^r}+\mathbb{Z} \right).
\end{align}
Note that in above definitions, we require that $m \in \mathbb{Z}_{\geq 1}$, $r \in \mathbb{Z}_{\geq 1}$ and $a,b,c \in \mathbb{Z}$ with $b$ odd. One can verify immediately that these are indeed finite quadratic modules. Also note that when $b$ is even in \eqref{eq:DefD}, then $\underline{D}_{2^r}^{a,b,c}$ is degenerate.
\end{examp}

\begin{lemm}
\label{lemm:DabcIsoToBorC}
Let $b$ be an odd integer and $r$ be a positive integer. If $a$ and $c$ are both odd integers, then $\underline{D}_{2^r}^{a,b,c}$ defined in \eqref{eq:DefD} is isometrically isomorphic to $\underline{B}_{2^r}$; otherwise it is isomorphic to $\underline{C}_{2^r}$.
\end{lemm}

\begin{proof}
Note that isomorphisms of $\mathbb{Z}/2^r\mathbb{Z} \times \mathbb{Z}/2^r\mathbb{Z}$ as an abelian group are exactly isomorphisms of that as a $\mathbb{Z}/2^r\mathbb{Z}$-module, and that $\mathbb{Z}/2^r\mathbb{Z} \times \mathbb{Z}/2^r\mathbb{Z}$ is a free $\mathbb{Z}/2^r\mathbb{Z}$-module. Hence we use invertible matrices with entries in $\mathbb{Z}/2^r\mathbb{Z}$ to represent isomorphisms of $\mathbb{Z}/2^r\mathbb{Z} \times \mathbb{Z}/2^r\mathbb{Z}$: the $2 \times 2$ invertible matrix $A$ represents the isomorphism $(a,b)^{\mathrm{T}} \mapsto A\cdot(a,b) ^{\mathrm{T}}$, where $a,\,b \in \mathbb{Z}/2^r\mathbb{Z}$.

Suppose at least one of $a$ and $c$ are even. We shall give an isometric  isomorphism of $\underline{D}_{2^r}^{a,b,c}$ and $\underline{C}_{2^r}$. To this end, we define a ``valuation'' $f$ on the ring $\mathbb{Z}/2^r\mathbb{Z}$ as follows. For nonzero $x\in \mathbb{Z}/2^r\mathbb{Z}$, if $x\in 2^j\mathbb{Z}/2^r\mathbb{Z}$ but $x\notin 2^{j+1}\mathbb{Z}/2^r\mathbb{Z}$, then $f(x)=j$; moreover, $f(0)=r$. We have $f(xy)=\min\{f(x)+f(y),\,r\}$. We also need the map $\sigma_a \colon (\mathbb{Z}/2^r\mathbb{Z})^\times \rightarrow (\mathbb{Z}/2^r\mathbb{Z})^\times$, $t \mapsto at+t^{-1}$, where $(\mathbb{Z}/2^r\mathbb{Z})^\times$ denotes the group of units, and $a$ satisfies $f(a)>0$. It is not hard to see that $\sigma_a$ is well-defined. We assert that this map is an injection, hence a bijection since $(\mathbb{Z}/2^r\mathbb{Z})^\times$ has finite cardinality. To prove this, let $t_1$ and $t_2$ be in $(\mathbb{Z}/2^r\mathbb{Z})^\times$ such that $at_1+t_1^{-1}=at_2+t_2^{-1}$. Then $at_1t_2(t_1-t_2)=t_1-t_2$. Apply the function $f$ to both sides; we obtain $t_1=t_2$. This concludes the proof of injectivity. Denoted by $\tau_a$ the inverse function of $\sigma_a$. Now one can verify that the isomophism given by the matrix ($a$, $b$, $1$ denote their image under the canonical map from $\mathbb{Z}$ onto $\mathbb{Z}/2^r\mathbb{Z}$)
\begin{equation*}
\begin{pmatrix}
a & (\tau_{ac}(b))^{-1} \\
1& c\cdot\tau_{ac}(b)
\end{pmatrix}
\end{equation*}
is indeed an isometric isomorphism from $\underline{D}_{2^r}^{a,b,c}$ onto $\underline{C}_{2^r}$.

On the other hand, suppose both of $a$ and $c$ are odd. We use induction on $r$ to show that $\underline{D}_{2^r}^{a,b,c}$ is isometrically isomorphic to $\underline{B}_{2^r}$ with an isomorphism given by a matrix whose entries on the main diagonal are non-units in the ring  $\mathbb{Z}/2^r\mathbb{Z}$, and entries off the main diagonal are units. For $r=1$, the desired isomorphism is given by the matrix whose entries on the main diagonal are $0$, and entries off the main diagonal are $1$. For the induction step, assume that we have found an isometric isomorphism from $\underline{D}_{2^r}^{a,b,c}$ onto $\underline{B}_{2^r}$, which is given by the matrix
\begin{equation*}
\begin{pmatrix}
u & v \\
w& z
\end{pmatrix}
\end{equation*}
where $u$ and $z$ are non-units in $\mathbb{Z}/2^r\mathbb{Z}$, and $v$ and $w$ are units in $\mathbb{Z}/2^r\mathbb{Z}$. By abuse of language, we still use $u$ ($v,\,w,\,z$ respectively) to denote an element of the inverse image of $u$ ($v,\,w,\,z$ respectively) under the map $\mathbb{Z}/2^{r+1}\mathbb{Z} \rightarrow \mathbb{Z}/2^{r}\mathbb{Z}$, $t+2^{r+1}\mathbb{Z} \mapsto t+2^{r}\mathbb{Z}$. (In the followings, we also use these letters to represent elements of the inverse image of corresponding values under the canonical map $\mathbb{Z}\rightarrow \mathbb{Z}/2^{r+1}\mathbb{Z}$.) We want to find the desired isomorphism from $\underline{D}_{2^{r+1}}^{a,b,c}$ onto $\underline{B}_{2^{r+1}}$ among the following matrices:
\begin{equation*}
\begin{pmatrix}
u + 2^rj_0 & v+2^rj_1 \\
w + 2^rj_2& z+2^rj_3
\end{pmatrix}
\end{equation*}
with $j_0,\,j_1,\,j_2,\,j_3 \in \{0+2^{r+1}\mathbb{Z},\,1+2^{r+1}\mathbb{Z}\}$. For any choice of $j_0,\,j_1,\,j_2,\,j_3$, the entries on the main diagonal of the above matrix are non-units and those off the main diagonal are units in $\mathbb{Z}/2^{r+1}\mathbb{Z}$. Particularly, this matrix is invertible, hence gives an isomorphism from $\underline{D}_{2^{r+1}}^{a,b,c}$ onto $\underline{B}_{2^{r+1}}$ as abelian groups. So we shall choose $j_0,\,j_1,\,j_2,\,j_3$ such that the quadratic maps of $\underline{D}_{2^{r+1}}^{a,b,c}$ and $\underline{B}_{2^{r+1}}$ are preserved by the corresponding map, that is,
\begin{gather*}
(u+2^rj_0)^2+(u+2^rj_0)(w+2^rj_2)+(w+2^rj_2)^2 \equiv a \mod 2^{r+1}, \\
2(u+2^rj_0)(v+2^rj_1)+(u+2^rj_0)(z+2^rj_3) \notag \\
+(v+2^rj_1)(w+2^rj_2)+2(w+2^rj_2)(z+2^rj_3) \equiv b \mod 2^{r+1}, \\
(v+2^rj_1)^2+(v+2^rj_1)(z+2^rj_3)+(z+2^rj_3)^2 \equiv c \mod 2^{r+1},
\end{gather*}
which are equivalent to
\begin{align}
\label{ProofOfDabcIsoB2rC2rEqu1}
2^r(j_0w+j_2u) &\equiv a-(u^2+uw+w^2) \mod 2^{r+1}, \\
\label{ProofOfDabcIsoB2rC2rEqu2}
2^r(j_0z+j_1w+j_2v+j_3u) &\equiv b-(2uv+uz+vw+2wz) \mod 2^{r+1}, \\
\label{ProofOfDabcIsoB2rC2rEqu3}
2^r(j_1z+j_3v) &\equiv c-(v^2+vz+z^2) \mod 2^{r+1}.
\end{align}
Since the right-hand sides of the above three congruences are both congruent to $0$ or $2^r$ by induction hypothesis, $u$, $z$ (considered as in $\mathbb{Z}$)are divisible by $2$, and $v$, $w$ (considered as in $\mathbb{Z}$)are not divisible by $2$, we can sovle for $j_0$ and $j_3$ from \eqref{ProofOfDabcIsoB2rC2rEqu1} and \eqref{ProofOfDabcIsoB2rC2rEqu3}. Then we can solve for $j_1$ and $j_2$ from \eqref{ProofOfDabcIsoB2rC2rEqu2}. Namely, the desired $j_0,\,j_1,\,j_2,\,j_3$ exist, which completes the induction step, hence the whole proof.
\end{proof}

We now proceed to show that modules given in Definition \ref{deff:FQMABC} actually serve as cornerstones of all finite quadratic modules. This result can be found in, for example, \cite[Proposition 2.11]{Stromberg2013}. We give a new proof here.

First of all, we consider the primary decomposition.
\begin{lemm}
\label{lema:PriDecomOfFQM}
Let $\underline{M}=(M,Q)$ be a finite quadratic module. For a prime $p$, denote by $M(p)$ the set of all elements in $M$ whose order is a power of $p$. Then, $M(p)$, with the quadratic form inherited from $\underline{M}$, is itself a finite quadratic module. Moreover, we have an orthogonal decomposition $M=\bigoplus_p M(p)$ where all but finitely many $M(p)$'s are zero. This decomposition is unique: if $M=\bigoplus_p N_p$ where $N_p$ contains only elements whose order are  powers of $p$, then $N_p=M(p)$.
\end{lemm}
\begin{proof}
Denote $B(x,y)=Q(x+y)-Q(x)-Q(y)$. It is obvious that $M(p)$ is a $\mathbb{Z}$-submodule of $M$, and by the primary decomposition theorem of finite abelian groups we have $M=\bigoplus_p M(p)$. For primes $p \neq q$, and $x\in M(p)$, $y\in M(q)$, suppose that $p^\alpha x = 0$ and $q^\beta y = 0$. Since $p^\alpha$ and $q^\beta$ are co-prime, there exist $a,b \in \mathbb{Z}$ such that $ap^\alpha+bq^\beta = 1$. Hence $B(x,y)=(ap^\alpha+bq^\beta)B(x,y)=B(ap^\alpha x,y)+B(x,bq^\beta y)=0$. That is, the direct sum $\bigoplus_p M(p)$ is orthogonal. The uniqueness also follows from the primary decomposition theorem of finite abelian groups. It remains to prove that all $M(p)$ are finite quadratic modules, i.e., the bilinear map restricted on each $M(p)$ is nondegenerate. This follows from the fact that $\underline{M}$ is nondegenerate and that the sum $\bigoplus_p M(p)$ is orthogonal.
\end{proof}

Secondly, we deal with $p$-primary finite quadratic modules, i.e., finite quadratic modules whose order is a positive power of $p$. The case $p=2$ and $p>2$ are very different, so we treat them separately in the following two lemmas.

\begin{lemm}
\label{lemm:pGroupDecom}
Any $p$-primary finite quadratic module where $p$ is an odd prime is isometrically isomorphic to an orthogonal direct sum of modules of the form \eqref{eq:DefAp}.
\end{lemm}
\begin{proof}
Let $\underline{M}=(M,Q)$ be a $p$-primary finite quadratic module and $B$ the associated bilinear map $B(x,y)=Q(x+y)-Q(x)-Q(y)$. By the theorem of cyclic decompositions of $p$-primary finite abelian groups, without loss of generality, we can assume that $M=\mathbb{Z}/p^{r_1}\mathbb{Z} \times \dots \times \mathbb{Z}/p^{r_t}\mathbb{Z}$ with $r_1 \geq r_2 \geq \dots r_t \geq 1$. For $1 \leq j \leq t$, set $\mathbf{1}_j=(\dots, 1+p^{r_j}\mathbb{Z}, \dots)$, where the $j^\prime$-th coordinate equals $0+p^{r_{j^\prime}}\mathbb{Z}$ if $1 \leq j^\prime \leq t$ and $j^\prime \neq j$. It can be verified that $Q(\mathbf{1}_j)=\frac{a_{j,j}}{p^{\beta_{j,j}}}+\mathbb{Z}$ and $Q(\mathbf{1}_i + \mathbf{1}_j)=\frac{a_{i,j}}{p^{\beta_{i,j}}}+\mathbb{Z}$ with $a_{i,j}$ integers not divisible by $p$ and $\beta_{i,j}$ nonnegative integers not exceeding $\max(r_i,r_j)$.

We now proceed by induction on $t$. The base case $t=1$ is clear: since $\underline{M}$ is nondegenerate we have $\beta_{1,1}=r_1$; hence $\underline{M}$ is isomorphic to $\underline{A}_{p^{r_1}}^{a_{1,1}}$. Assume that the assertion holds for the case $t=t_0$; we shall prove it for the case $t=t_0+1$. Choose a pair $(i_0,j_0)$ such that the quantity $\max(r_{i_0},r_{j_0})-\beta_{i_0,j_0}$ is the least among all  quantities $\max(r_{i},r_{j})-\beta_{i,j}$. Then we have $\max(r_{i_0},r_{j_0})-\beta_{i_0,j_0}=0$ since the negation $\max(r_{i_0},r_{j_0})-\beta_{i_0,j_0}>0$ contradicts the assumption that $\underline{M}$ is nondegenerate (for example, the element $p^{r_1-1}\mathbf{1}_1$ is orthogonal to the whole $M$ in this case.). Denote by $G$ the subgroup generated by $\mathbf{1}_{j_0}$ (in the case $i_0=j_0$) or by $\mathbf{1}_{i_0}+\mathbf{1}_{j_0}$ (in the case $i_0 > j_0$), both of which are isometrically isomorphic to $\underline{A}_{p^{r_{j_0}}}^{a_{i_0,j_0}}$. Apply Lemma \ref{lemm:BasicLemmaOfStrongliNondeg}(1) and Corollary \ref{coro:FQMStroNondeg}; we have an orthogonal direct sum $M=G \oplus G^\perp$. The $\mathbb{Z}$-module $G^\perp$ is again a finite quadratic module by Lemma \ref{lemm:BasicLemmaOfStrongliNondeg}(2). Moreover any decomposition of $G^\perp$ into cyclic groups has $t_0$ summands by elementary theory of finite abelian groups. Hence we conclude the proof by applying the induction hypothesis to $G^\perp$.
\end{proof}

\begin{lemm}
\label{lemm:2GroupDecom}
Any $2$-primary finite quadratic module is isometrically isomorphic to an orthogonal direct sum of modules of the form \eqref{eq:DefA2}, \eqref{eq:DefB2}, or \eqref{eq:DefC2}.
\end{lemm}
\begin{proof}
Similarly as in the proof of Lemma \ref{lemm:pGroupDecom}, let $\underline{M}=(M,Q)$ be a $2$-primary finite quadratic module and $B$ the associated bilinear map. Assume that $M=\mathbb{Z}/2^{r_1}\mathbb{Z} \times \dots \times \mathbb{Z}/2^{r_t}\mathbb{Z}$ with $r_1 \geq r_2 \geq \dots r_t \geq 1$ without loss of generality. Again, for $1 \leq j \leq t$, set $\mathbf{1}_j=(\dots, 1+2^{r_j}\mathbb{Z}, \dots)$, where the $j^\prime$-th coordinate equals $0+2^{r_{j^\prime}}\mathbb{Z}$ if $1 \leq j^\prime \leq t$ and $j^\prime \neq j$. At this point, something different from the case $p>2$ happens: we can verify that $Q(\mathbf{1}_j)=\frac{a_{j,j}}{2^{\beta_{j,j}}}+\mathbb{Z}$ and $Q(\mathbf{1}_i + \mathbf{1}_j)=\frac{a_{i,j}}{2^{\beta_{i,j}}}+\mathbb{Z}$ with $a_{i,j}$ integers not divisible by $2$ and $\beta_{i,j}$ nonnegative integers not exceeding $\max(r_i,r_j)+1$ instead of the quantity $\max(r_i,r_j)$ in the case $p>2$.

Proceed by strong induction on $t$. The base case $t=1$ is clear as in the proof of Lemma \ref{lemm:pGroupDecom}: $\underline{M}$ must be isometrically isomorphic to $\underline{A}_{2^{r_1}}^{a_{1,1}}$. Assume that the statement is true for the case $t<t_0$; we aim to prove it for the case $t=t_0$. Choose a pair $(i_0,j_0)$ such that the quantity $\max(r_{i_0},r_{j_0})+1-\beta_{i_0,j_0}$ is the least among all  quantities $\max(r_{i},r_{j})+1-\beta_{i,j}$. There are exactly two cases $\max(r_{i_0},r_{j_0})+1-\beta_{i_0,j_0} = 0$ or $1$ to consider, since other cases contradict the assumption that $\underline{M}$ is nondegenerate (for example, the element $2^{r_1-1}\mathbf{1}_1$ is orthogonal to the whole $M$ in these cases.). We will prove in the next paragraph that in both cases we can find a submodule $G$ isometrically isomorphic to one of the modules $\underline{A}_{2^r}^a$, or $\underline{B}_{2^r}$, or $\underline{C}_{2^r}$. The rest of the proof is almost the same as the last few sentences of the proof of Lemma \ref{lemm:pGroupDecom}.

For the case $\max(r_{i_0},r_{j_0})+1-\beta_{i_0,j_0} = 0$, we proceed similarly as in the proof of Lemma \ref{lemm:pGroupDecom}: define $G$ as the subgroup generated by $\mathbf{1}_{j_0}$ (in the case $i_0=j_0$) or by by $\mathbf{1}_{i_0}+\mathbf{1}_{j_0}$ (in the case $i_0 > j_0$), each of which is isometrically isomorphic to $\underline{A}_{2^{r_{j_0}}}^{a_{i_0,j_0}}$. For another case $\max(r_{i_0},r_{j_0})+1-\beta_{i_0,j_0} = 1$, we have $\beta_{i,j} \leq \max(r_{i},r_{j})$ for any $i,j$. Thus $B(2^{r_1-1}\mathbf{1}_1, \mathbf{1}_1)=0+\mathbb{Z}$. Howerver, since $B$ is nondegenerate, there exists a $j_1 \neq 1$ such that $B(2^{r_1-1}\mathbf{1}_1, \mathbf{1}_{j_1}) \neq 0+\mathbb{Z}$. Define $G$ as the subgroup of $M$ generated by $\mathbf{1}_1$ and $\mathbf{1}_{j_1}$. We assert that $G$ is isometrically isomorphic to $\underline{B}_{2^{r_1}}$ or $\underline{C}_{2^{r_1}}$. First of all, we show that $r_{j_1}=r_1$. Since $B(\mathbf{1}_1, \mathbf{1}_{j_1})=Q(\mathbf{1}_1 + \mathbf{1}_{j_1})-Q(\mathbf{1}_1)-Q(\mathbf{1}_{j_1})$, we have $B(\mathbf{1}_1, \mathbf{1}_{j_1}) = \frac{b}{2^r}+\mathbb{Z}$ for some $r \leq r_1$ and some odd integer $b$. However, $B(2^{r_1-1}\mathbf{1}_1, \mathbf{1}_{j_1}) \neq 0+\mathbb{Z}$ implies that $r=r_1$. Moreover, we have $Q(\mathbf{1}_1)= \frac{a}{2^{r_1^\prime}}+\mathbb{Z}$ and $Q(\mathbf{1}_{j_1})= \frac{c}{2^{r_{j_1}^\prime}}+\mathbb{Z}$ for some $r_1^\prime \leq r_1$, $r_{j_1}^\prime \leq r_{j_1} \leq r_1$ and some odd integers $a$ and $c$. A direct calculation shows that
\begin{multline}
\label{eq:valueBInProofOf2Primary}
B(x_1\mathbf{1}_1+y_1\mathbf{1}_{j_1}, x_2\mathbf{1}_1+y_2\mathbf{1}_{j_1}) \\
 = \frac{2^{r_1-r_1^\prime+1}ax_1x_2+b(x_1y_2+y_1x_2)+2^{r_1-r_{j_1}^\prime+1}cy_1y_2}{2^{r_1}}+\mathbb{Z}
\end{multline}
for integers $x_1$, $y_1$, $x_2$, $y_2$. We now seek for integers $x_1$ and $y_1$ such that $x_1\mathbf{1}_1+y_1\mathbf{1}_{j_1} \perp G$, i.e., $B(x_1\mathbf{1}_1+y_1\mathbf{1}_{j_1},\mathbf{1}_1)=0$ and $B(x_1\mathbf{1}_1+y_1\mathbf{1}_{j_1},\mathbf{1}_{j_1})=0$. Let $x_2=1$ and $y_2=0$ in \eqref{eq:valueBInProofOf2Primary}; then $B(x_1\mathbf{1}_1+y_1\mathbf{1}_{j_1},\mathbf{1}_1)=0$ if and only if $2^{r_1} \mid 2^{r_1-r_1^\prime+1}ax_1+by_1$. Similarly, let $x_2=0$ and $y_2=1$; then $B(x_1\mathbf{1}_1+y_1\mathbf{1}_{j_1},\mathbf{1}_{j_1})=0$ if and only if $2^{r_1} \mid bx_1+2^{r_1-r_{j_1}^\prime+1}cy_1$. Hence, $x_1\mathbf{1}_1+y_1\mathbf{1}_{j_1} \perp G$ if and only if $2^{r_1} \mid \left(u2^{r_1-r_1^\prime+1}a+vb\right)x_1+\left(ub+v2^{r_1-r_{j_1}^\prime+1}c\right)y_1$ for any $u,\,v \in \mathbb{Z}$. Since the determinant of the matrix
\begin{equation*}
\begin{pmatrix}
2^{r_1-r_1^\prime+1}a & b \\
b & 2^{r_1-r_{j_1}^\prime+1}c
\end{pmatrix}
\end{equation*}
is odd, we obtain that $x_1\mathbf{1}_1+y_1\mathbf{1}_{j_1} \perp G$ if and only if $x_1$ and $y_1$ are both multiples of $2^{r_1}$. Consequently, we have $r_{j_1}=r_1$. Secondly, we can prove that $G$, with the quadratic map inherited from $\underline{M}$, is a finite quadratic module. This also follows from the equivalence just proved. Finally, we have $G$ is isometrically isomorphic to $\underline{B}_{2^{r_1}}$ or $\underline{C}_{2^{r_1}}$, which follows from Lemma \ref{lemm:DabcIsoToBorC} and the fact that $G$ is isometrically isomorphic to $\underline{D}_{2^{r_1}}^{2^{r_1-r_1^\prime}a,b,2^{r_1-r_{j_1}^\prime}c}$.
\end{proof}

Now we can state the structure theorem of finite quadratic modules:
\begin{thm}
\label{thm:JordanDecompositionOfFQM}
Every finite quadratic module can be decomposed as an orthogonal direct sum of indecomposable modules defined in Definition \ref{deff:FQMABC}.
\end{thm}
\begin{proof}
This follows immediately from Lemma \ref{lema:PriDecomOfFQM} , Lemma \ref{lemm:pGroupDecom}, and Lemma \ref{lemm:2GroupDecom}.
\end{proof}
Note that the decomposition of the zero module is an empty sum by convention.

\section{Discriminant modules of even lattices}
\label{sec:Discriminant modules of even lattices}

In the last section, we illustrated that we have a finite quadratic module associated to any even integral lattice---the discriminant module of the given lattice. We now prove the converse: every finte quadratic module is isometrically isomorphic to the discriminant module of an even integral lattice, which is important for our later paper on Jacobi forms of lattice index. This result was first discovered and proved in \cite[Theorem 6]{Wall1963}. We give a proof here by presenting a concrete lattice represented by its Gram matrix for any of the modules in Definition \ref{deff:FQMABC}. In this section, we give an even lattice of the least rank, while in the next section, we give a positive definite even lattice of the least rank, for each of these finite quadratic modules.

We begin with modules $\underline{A}_{p^r}^a$ (See Definition \ref{deff:FQMABC}.).
\begin{lemm}
\label{lemm:HowToFindEvenLatticeForApr}
Let $p$ be an odd prime, $r$ be a positive integer, and $a \in \mathbb{Z}$ with $(p,a)=1$. Let $\underline{L}=(L,B)$ be an (nondegenerate) even lattice of rank $n \in \mathbb{Z}_{\geq 1}$. Then the modules $D_{\underline{L}}$ (See Example \ref{examp:DiscModule} for definition) and $\underline{A}_{p^r}^{a}$ are isometrically isomorphic if and only if there exists a basis of $L$ such that the Gram matrix of $\underline{L}$ with respect to this basis is of the form $\mathop{\mathrm{diag}}\left(p^r, 1, \dots, 1\right) \cdot S$, where $S=(s_{ij})_{1 \leq i,j \leq n}$ is an invertible integral matrix (i.e., $S \in GL_n(\mathbb{Z})$) satisfying
\begin{enumerate}
\item $s_{11},\dots,s_{nn} \in 2\mathbb{Z}$,
\item The Legendre symbol $\legendre{s_{11}/2}{p}=\legendre{a}{p}$,
\item $p^rs_{1j}=s_{j1}$ for $2\leq j\leq n$ and $s_{ij}=s_{ji}$ for $2 \leq i,j \leq n$.
\end{enumerate}
\end{lemm}

\begin{proof}
We denote by $Q$ the quadratic map of $D_{\underline{L}}$, and $B_Q$ the associated bilinear map.

Suppose that $D_{\underline{L}}$ and $\underline{A}_{p^r}^{a}$ are isomorphic. Then there is a basis $(e_1,e_2,\dots,e_n)$ of $L$ such that $(\frac{1}{p^r}e_1,e_2,\dots,e_n)$ is a basis of $L^\sharp$. Let $G$ be the Gram matrix of $\underline{L}$ with respect to $(e_1,e_2,\dots,e_n)$. Since $G^{-1}(e_1,e_2,\dots,e_n)^{\mathrm{T}}$ is also a basis of $L^\sharp$ by Proposition \ref{propBasisPropOnDualLattice}, there exists an $S \in GL_n(\mathbb{Z})$ such that $S\cdot G^{-1}(e_1,e_2,\dots,e_n)^{\mathrm{T}} = (\frac{1}{p^r}e_1,e_2,\dots,e_n)^{\mathrm{T}}$, i.e., $G=\mathop{\mathrm{diag}}(p^r,1,\dots,1)\cdot S$. We claim that this $S=(s_{ij})_{1 \leq i,j \leq n}$ satisfies the required three conditions. First of all, $s_{ii}$ are all even integers since $\underline{L}$ is an even lattice. Secondly, we have $\legendre{s_{11}/2}{p}=\legendre{a}{p}$ by Remark \ref{rema:SimpleModIso} and the fact that $Q(x_1\frac{1}{p^r}e_1+x_2e_2+\dots x_ne_n+L)=\frac{s_{11}/2}{p^r}x_1^2 +\mathbb{Z}$ with $x_1,\dots,x_n \in \mathbb{Z}$. Finally, the third condition on $S$ holds, because $G$ is symmetric.

Conversely, suppose that there is a basis $(e_1, e_2, \dots, e_n)$ of $L$ such that the Gram matrix $G$ of $\underline{L}$ with respect to this basis can be decomposed as $G=\mathop{\mathrm{diag}}\left(p^r, 1, \dots, 1\right) \cdot S$, where $S=(s_{ij})_{1 \leq i,j \leq n} \in GL_n(\mathbb{Z})$ satisfies the three conditions listed in the lemma. Then $G^{-1}(e_1, e_2, \dots, e_n)^{\mathrm{T}}$, hence $S\cdot G^{-1}(e_1, e_2, \dots, e_n)^{\mathrm{T}}$, is a basis of $L^\sharp$. It follows that  $Q(x_1\frac{1}{p^r}e_1+x_2e_2+\dots x_ne_n+L)=\frac{s_{11}/2}{p^r}x_1^2 +\mathbb{Z}$ for  $x_1,\dots,x_n \in \mathbb{Z}$ using \eqref{eq:RelationBetQandBQ}. So $D_{\underline{L}}$ is isomorphic to $\underline{A}_{p^r}^{s_{11}/2}$, and hence isomorphic to $\underline{A}_{p^r}^{a}$, since $\legendre{s_{11}/2}{p}=\legendre{a}{p}$.
\end{proof}

\begin{thm}
\label{thm:LatticeOfLeastRankOfFQMApr}
Let $p$ be an odd prime, $r$ be a positive integer, and $a \in \mathbb{Z}$ with $(p,a)=1$.
\begin{enumerate}
\item If $\legendre{a}{p}=1$, then the discriminant module of the even nondegenerate lattice given by the Gram matrix
\begin{equation}
\label{eq:GramMatrixNondegLatticeOfApr1}
\begin{pmatrix}
2p^r & p^r \\
p^r & \frac{1}{2}\left(p^r-\legendre{-1}{p}^r\right)
\end{pmatrix}
\end{equation}
is isometrically isomorphic to $\underline{A}_{p^r}^{a}$. Moreover, there is no even nondegenerate lattice of a smaller rank with this property.
\item If $\legendre{a}{p}=-1$ and $r$ is even, we construct an even nondegenerate lattice as follows: choose another odd prime $q \neq p$ such that $\legendre{q}{p}=-1$, and let $v$ be a solution of the system of congruences
\begin{equation*}
p^rv^2 \equiv 1 \mod q, \qquad v^2 \equiv 1 \mod 4.
\end{equation*}
Define the required lattice via the Gram matrix
\begin{equation}
\label{eq:GramMatrixNondegLatticeOfApr2}
\begin{pmatrix}
2qp^r & p^rv \\
p^rv & \frac{p^rv^2-1}{2q}
\end{pmatrix}
.
\end{equation}
Then the discriminant module of this lattice is isometrically isomorphic to $\underline{A}_{p^r}^{a}$. Moreover, there is no even nondegenerate lattice of a smaller rank with this property.
\item If  $\legendre{a}{p}=-1$, $r$ is odd, and $ p \equiv 3 \mod 4$, then the discriminant module of the even nondegenerate lattice given by the Gram matrix
\begin{equation}
\label{eq:GramMatrixNondegLatticeOfApr3}
\begin{pmatrix}
-2p^r & p^r \\
p^r & -\frac{1}{2}\left(p^r+1\right)
\end{pmatrix}
\end{equation}
is isometrically isomorphic to $\underline{A}_{p^r}^{a}$. Moreover, there is no even nondegenerate lattice of a smaller rank with this property.
\item If $\legendre{a}{p}=-1$, $r$ is odd, and $ p \equiv 1 \mod 4$, then there exist no even nondegenerate lattices of rank $1$ or $2$ whose discriminant module is isometrically isomorphic to $\underline{A}_{p^r}^{a}$. Moreover, there is an even nondegenerate lattice of rank $4$ whose discriminant module is isomorphic to $\underline{A}_{p^r}^{a}$. To give such a lattice, we first choose an odd prime $q$ such that $q \equiv 2 \mod 3$ and $\legendre{q}{p}=-1$. Then we let $v$ be a solution of the system of congruences
\begin{equation*}
p^rv^2 \equiv -\frac{4q+1}{3} \mod q, \qquad v^2 \equiv -\frac{4q+1}{3} \mod 4,
\end{equation*}
and set $x=\frac{1}{4q}(p^rv^2+\frac{4q+1}{3})$. The lattice is then given by the Gram matrix
\begin{equation}
\label{eq:GramMatrixNondegLatticeOfApr4}
\begin{pmatrix}
2qp^r & p^rv & 0 & 0 \\
p^rv & 2x & 1 & 0 \\
0 & 1 & 2 & 1 \\
0 & 0 & 1 & 2
\end{pmatrix}
\end{equation}
\end{enumerate}
\end{thm}

\begin{rema}
In the fourth part we did not tell anything about whether there is a lattice of rank $3$ satisfying the required condition. It turned out that such lattices do not exist by Milgram's formula (See, for instance, \cite[Appendix 4]{MH1973}). We will get back to this matter later in Remark \ref{rema:CompletionOfThmApr}.
\end{rema}

\begin{proof}[Proof of Theorem \ref{thm:LatticeOfLeastRankOfFQMApr}]
First of all, note that there is no even nendegenerate lattice of rank $1$ whose discriminant module is isomorphic to $\underline{A}_{p^r}^{a}$, since such discriminant module must have an even cardinality. As a corollary, the statements in Part 1 and Part 3 follows, using Lemma \ref{lemm:HowToFindEvenLatticeForApr}.

Secondly, we prove the statement in Part 2. Since $r$ is even, $v$ must exist, and $\frac{p^rv^2-1}{2q}$ must be an even integer. Now the statement also follows from Lemma \ref{lemm:HowToFindEvenLatticeForApr}.

Finally, we prove the statement in Part 4. There exists no even nondegenerate lattice of rank $2$ whose discriminant module is $\underline{A}_{p^r}^{a}$. This can be proved by contradicition. Assume that there exists such lattice $\underline{L}$. By Lemma \ref{lemm:HowToFindEvenLatticeForApr}, the Gram matrix of $\underline{L}$ with respect to some basis is $\mathop{\mathrm{diag}}(p^r,1)\cdot S$, with S being
\begin{equation*}
\begin{pmatrix}
2x & v \\
p^rv & 2y
\end{pmatrix}
\end{equation*}
for some $x, y, v \in \mathbb{Z}$ with $\legendre{x}{p}=-1$ and $4xy-p^rv^2=\pm 1$. By reduction mod $4$ in the latter formula, only the minus symbol is possible. If there is an odd prime factor $q$ of $x$ such that $\legendre{q}{p}=-1$, then $\legendre{-1}{q}\cdot\legendre{p^r}{q}\cdot\legendre{v^2}{q}=\legendre{-1}{q}$, which contradicts the quadratic reciprocity law. Therefore, $2 \mid x$, and $\legendre{2}{p}=-1$. It follows that $8 \mid p^rv^2-1$ and $p \equiv 5 \mod 8$, which contradicts each other. It remains to prove that the given lattice \eqref{eq:GramMatrixNondegLatticeOfApr4} actually produces the required module. The prime $q$ must exist. Moreover, we have $\legendre{-3}{q}=\legendre{-1}{q}\cdot\legendre{3}{q}=(-1)^{\frac{q-1}{2}}\cdot\legendre{q}{3}\cdot(-1)^{\frac{q-1}{2}}=-1$. Let $m = -\frac{4q+1}{3}$, so $\legendre{-3}{q}\cdot\legendre{m}{q}=\legendre{4q+1}{q}$, which implies that $\legendre{m}{q}=-1$. It follows that $v$ exists and $x$ is an integer. Now we can verify straightforwardly that the Gram matrix \eqref{eq:GramMatrixNondegLatticeOfApr4} satisfies the condition in Lemma \ref{lemm:HowToFindEvenLatticeForApr}.
\end{proof}

Next we deal with the modules $\underline{A}_{2^r}^a$ (See Definition \ref{deff:FQMABC}.).
\begin{lemm}
\label{lemm:HowToFindEvenLatticeForA2r}
Let $r$ be a positive integer, and $a$ be an odd integer. Let $\underline{L}=(L,B)$ be an (nondegenerate) even lattice of rank $n \in \mathbb{Z}_{\geq 1}$. Then the modules $D_{\underline{L}}$ (See Example \ref{examp:DiscModule} for definition) and $\underline{A}_{2^r}^{a}$ are isometrically isomorphic if and only if there exists a basis of $L$ such that the Gram matrix of $\underline{L}$ with respect to this basis is of the form $\mathop{\mathrm{diag}}\left(2^r, 1, \dots, 1\right) \cdot S$, where $S=(s_{ij})_{1 \leq i,j \leq n}$ is an invertible integral matrix (i.e., $S \in GL_n(\mathbb{Z})$) satisfying
\begin{enumerate}
\item $s_{22},\dots,s_{nn} \in 2\mathbb{Z}$,
\item $s_{11} \notin 2\mathbb{Z}$ and $s_{11}a$ is a quadratic residue mod $2^{r+1}$,
\item $2^rs_{1j}=s_{j1}$ for $2\leq j\leq n$ and $s_{ij}=s_{ji}$ for $2 \leq i,j \leq n$.
\end{enumerate}
\end{lemm}
\begin{proof}
The proof is almost the same as that of Lemma \ref{lemm:HowToFindEvenLatticeForApr}. Just replace $p$ by $2$. We omit the details.
\end{proof}

To state the theorem for $\underline{A}_{2^r}^a$, we need first recall some basic facts in elementary number theory. Denote by $\left(\mathbb{Z}/2^{r+1}\mathbb{Z}\right)^\times$ the multiplicative group of all units in the quotient ring $\mathbb{Z}/2^{r+1}\mathbb{Z}$, where $r \geq 1$. Then the map
\begin{align*}
\mathbb{Z}/2\mathbb{Z} \times \mathbb{Z}/2^{r-1}\mathbb{Z} &\longrightarrow \left(\mathbb{Z}/2^{r+1}\mathbb{Z}\right)^\times \\
(n_0 + 2\mathbb{Z},n_1 + 2^{r-1}\mathbb{Z}) &\longmapsto (-1)^{n_0}5^{n_1} + 2^{r+1}\mathbb{Z}
\end{align*}
is a group isomorphism, where $\mathbb{Z}/2\mathbb{Z}$ and $ \mathbb{Z}/2^{r-1}\mathbb{Z}$ are regarded as the additive group of the ambient rings. Moreover, let $U$ be the group of squares of elements in $\left(\mathbb{Z}/2^{r+1}\mathbb{Z}\right)^\times$. Then by the above isomorphism, one can verify that $\left[\left(\mathbb{Z}/2^{r+1}\mathbb{Z}\right)^\times\mathbin{:}U\right]$ equals $2$ if $r=1$, and equals $4$ if $r>1$. The set $\{\pm 1 + 2^{r+1}\mathbb{Z}, \pm 5 + 2^{r+1}\mathbb{Z} \}$ forms a complete system of representatives of the quotient $\left(\mathbb{Z}/2^{r+1}\mathbb{Z}\right)^\times/U$, when $r>1$. Moreover, the subgroup $U$ can be described explicitly as the set of elements $5^{2n} + 2^{r+1}\mathbb{Z}$, where $n$ is nonnegative integers, which turns out to be the set of elements $8n+1+2^{r+1}\mathbb{Z}$, where $n \in \mathbb{Z}$. Thus, when we treat $\underline{A}_{2^r}^a$ for $r>1$, there are four cases to consider, according to which element $\pm 1 + 2^{r+1}\mathbb{Z}, \pm 5 + 2^{r+1}\mathbb{Z}$ that $a+ 2^{r+1}\mathbb{Z}$ is congruent to modulo $U$, i.e., according to which number $\pm 1, \pm 5$ that $a$ is congruence to modulo $8$. 

\begin{thm}
\label{thm:LatticeOfLeastRankOfFQMA2r}
Let $r$ be a positive integer greater than $1$, and $a$ be an odd integer. 
\begin{enumerate}
\item If  $a$ is congruent to $1$ (or to $-1$, respectively) modulo $8$, then the discriminant module of the even nondegenerate lattice given by the $1 \times 1$ Gram matrix $(2^r)$ (or the matrix $(-2^r)$, respectively) is isometrically isomorphic to $\underline{A}_{2^r}^a$.
\item If $a$ is congruent to $\varepsilon \cdot 5 $ modulo $8$, where $\varepsilon = \pm 1$, and $r$ is odd, we define an even nondegenerate lattice as follows: let $v$ be a solution of the congruence $2^{r-1}v^2 \equiv -\varepsilon \mod 5$, and set $x=\frac{2^{r-1}v^2+\varepsilon}{5}$; then the required lattice is given by the Gram matrix
\begin{equation}
\label{eq:GramMatrixNondegLatticeOfA2r2}
\begin{pmatrix}
\varepsilon \cdot 5\cdot 2^r & 2^rv & 0 \\
2^rv & \varepsilon \cdot 2x & 1 \\
0 & 1 & 2
\end{pmatrix}
.
\end{equation} 
The discriminant module of this lattice is isometrically isomorphic to $\underline{A}_{2^r}^a$. Moreover, there is no even nondegenerate lattice of a smaller rank with this property.
\item If $a$ is congruent to $\varepsilon \cdot 5$ modulo $8$, where $\varepsilon = \pm 1$, and $r$ is even, we define an even nondegenerate lattice as follows: let $v$ be a solution of the congruence $2^{r-1}v^2 \equiv -1 \mod 3$, and set $x=\frac{2^{r-1}v^2+1}{3}$; then the required lattice is given by the Gram matrix
\begin{equation}
\label{eq:GramMatrixNondegLatticeOfA2r3}
\begin{pmatrix}
-\varepsilon \cdot 3\cdot 2^r & 2^rv & 0 \\
2^rv & -\varepsilon \cdot2x & 1 \\
0 & 1 & -\varepsilon \cdot2
\end{pmatrix}
.
\end{equation} 
The discriminant module of this lattice is isometrically isomorphic to $\underline{A}_{2^r}^a$. Moreover, there is no even nondegenerate lattice of a smaller rank with this property.
\item If  $a$ is congruent to $1$ (or to $-1$, respectively) modulo $4$, then the discriminant module of the even nondegenerate lattice given by the $1 \times 1$ Gram matrix $(2)$ (or the matrix $(-2)$, respectively) is isometrically isomorphic to $\underline{A}_{2}^a$.
\end{enumerate}
\end{thm}

\begin{proof}
The statements in Part 1 and 4 can be verified straightforwardly.

Now proceed to prove the statements in Part 2 and 3. First of all, one can prove that there is no even nondegenerate lattice of rank $1$ whose discriminant module is isomorphic to $\underline{A}_{2^r}^a$. Secondly, we will prove that there is no such lattice of rank $2$ by contradiction. Assume that there is a such lattice $\underline{L}$ of rank $2$. By Lemma \ref{lemm:HowToFindEvenLatticeForA2r}, $\underline{L}$ has a Gram matrix (with respect to some basis) of the form
\begin{equation*}
\begin{pmatrix}
2^ra & 2^rv \\
2^rv & 2x
\end{pmatrix}
\end{equation*}
with $a,v,x \in \mathbb{Z}$ and $2ax-2^rv^2=\pm 1$. This is impossible, since the determinant must be even. Finally, we prove that the lattices given by \eqref{eq:GramMatrixNondegLatticeOfA2r2} and \eqref{eq:GramMatrixNondegLatticeOfA2r3} actually serve the purpose. Note that the congruence in Part 2 and in Part 3 are both solvable, hence $x$'s are integers in both cases. Then we can verify that \eqref{eq:GramMatrixNondegLatticeOfA2r2} and \eqref{eq:GramMatrixNondegLatticeOfA2r3} satisfy the conditions in Lemma \ref{lemm:HowToFindEvenLatticeForA2r}, which implies the desired isomorphisms.
\end{proof}

Finally, we deal with the modules $\underline{D}_{2^r}^{a,b,c}$ (For definition, see \eqref{eq:DefD}). Note that this includes the modules $\underline{B}_{2^r}$ and $\underline{C}_{2^r}$ (defined in Definition \ref{deff:FQMABC}).

\begin{lemm}
\label{lemm:HowToFindEvenLatticeForD2r}
Let $a,b,c$ be integers with $b$ odd, and let $r$ be a positive integer. Let $\underline{L}=(L,B)$ be an even nondegenerate lattice of rank $n \in \mathbb{Z}_{\geq 1}$. Then $D_{\underline{L}}$ and $\underline{D}_{2^r}^{a,b,c}$ are isometrically isomorphic if and only if $n \geq 2$, and there exists a basis of $L$ such that the Gram matrix of $\underline{L}$ with respect to this basis is of the form $\mathop{\mathrm{diag}}(2^r,2^r,1,\dots,1)\cdot S$, where $S=(s_{ij})_{1 \leq i,j \leq n} \in GL_n(\mathbb{Z})$ satisfies
\begin{enumerate}
\item $s_{33},\dots,s_{nn} \in 2\mathbb{Z}$,
\item $s_{11}$, $s_{22}$ are even with $s_{11}s_{22} \equiv 4ac \mod 8$, and $s_{12}=s_{21}$ are odd, 
\item $2^rs_{1j}=s_{j1}$ and $2^rs_{2j}=s_{j2}$  for $3 \leq j \leq n$, and $s_{ij}=s_{ji}$ for $3 \leq i,j \leq n$.
\end{enumerate}
\end{lemm}

\begin{proof}
Denote by $Q$ the quadratic map of $D_{\underline{L}}$, and $B_Q$ the associated bilinear map.

Suppose that $D_{\underline{L}}$ and $\underline{D}_{2^r}^{a,b,c}$ are isomorphic. Then there is a basis $(e_1, e_2, \dots,e_n)$ of $L$ such that $(\frac{1}{2^r}e_1, \frac{1}{2^r}e_2, e_3, \dots,e_n)$ is a basis of $L^\sharp$. Hence $n \geq 2$. Let $G$ be the Gram matrix of $\underline{L}$ with respect to this basis. Then there is an $S \in GL_n(\mathbb{Z})$ such that $S\cdot G^{-1}=\mathop{\mathrm{diag}}(\frac{1}{2^r},\frac{1}{2^r},1, \dots, 1)$, i.e., $G=\mathop{\mathrm{diag}}(2^r, 2^r,1, \dots, 1)\cdot S$. We shall show that this $S$ satisfies the desired conditions. Since $\underline{L}$ is even we have $s_{33}, \dots, e_{nn} \in 2\mathbb{Z}$. Since $B$ is symmetric, we have $2^rs_{1j}=s_{j1}$ and $2^rs_{2j}=s_{j2}$  for $3 \leq j \leq n$, and $s_{ij}=s_{ji}$ for $3 \leq i,j \leq n$. By \eqref{eq:RelationBetQandBQ}, we have $Q(x_1\frac{1}{2^r}e_1+x_2\frac{1}{2^r}e_2+\dots+x_ne_n+L)=\frac{s_{11}x_1^2+2s_{12}x_1x_2+s_{22}x_2^2}{2^{r+1}}+\mathbb{Z}$. We must have $s_{11}$ is even, since otherwise the submodule generated by $\frac{1}{2^r}e_1+L$ would itself be a finite quadratic module, which implies that $\underline{D}_{2^r}^{a,b,c}$ has an orthogonal direct sum decomposition by Lemma \ref{lemm:BasicLemmaOfStrongliNondeg}. Likewise, we have $s_{22}$ is even. On the other hand, $s_{12}=s_{21}$ must be odd, since $B_Q$ is nondegenerate. Thus, $D_{\underline{L}}$ is isometrically isomorphic to $\underline{D}_{2^r}^{s_{11}/2,s_{12},s_{22}/2}$. Using Lemma \ref{lemm:DabcIsoToBorC},  we deduce that $s_{11}s_{22} \equiv 4ac \mod 8$.

Conversely, suppose that $L$ has a basis $(e_1, e_2, \dots,e_n)$ such that the Gram matrix of $\underline{L}$ with respect to this basis is of the form $\mathop{\mathrm{diag}}(2^r,2^r,1,\dots,1)\cdot S$, where $S=(s_{ij})_{1 \leq i,j \leq n} \in GL_n(\mathbb{Z})$ satisfies the conditions listed in the lemma. Then $(\frac{1}{2^r}e_1, \frac{1}{2^r}e_2, e_3, \dots,e_n)$ is a basis of $L^\sharp$, by Proposition \ref{propBasisPropOnDualLattice}. Using \eqref{eq:RelationBetQandBQ} and the conditions on $S$, we obtain $Q(x_1\frac{1}{2^r}e_1+x_2\frac{1}{2^r}e_2+\dots+x_ne_n+L)=\frac{s_{11}/2\cdot x_1^2+s_{12}x_1x_2+s_{22}/2\cdot x_2^2}{2^{r}}+\mathbb{Z}$. Thus $D_{\underline{L}}$ is isometrically isomorphic to $\underline{D}_{2^r}^{s_{11}/2,s_{12},s_{22}/2}$. Again by Lemma \ref{lemm:DabcIsoToBorC}, we have $D_{\underline{L}}$ is isometrically isomorphic to $\underline{D}_{2^r}^{a,b,c}$.
\end{proof}

\begin{thm}
\label{thm:LatticeOfLeastRankOfFQMBrCr}
Let $r$ be a positive integer, and set $\varepsilon=(-1)^{r+1}$.
\begin{enumerate}
\item Let $v$ be a solution of the congruence $2^rv^2 \equiv -\varepsilon \mod 3$, and set $x=\frac{2^rv^2+\varepsilon}{3}$. Then the discriminant module of the even nondegenerate lattice given by the Gram matrix
\begin{equation}
\label{eq:GramMatrixNondegLatticeOfBr}
\begin{pmatrix}
2^{r+1} & 2^r & 0 & 0 \\
2^r & 2^{r+1} & 2^rv & 0 \\
0 & 2^rv & 2x & 1 \\
0 & 0 & 1 & \varepsilon \cdot 2
\end{pmatrix}
\end{equation}
is isometrically isomorphic to $\underline{B}_{2^r}$. Moreover, there is no such lattice of smaller rank.
\item The discriminant module of the even nondegenerate lattice given by the following Gram matrix
\begin{equation}
\label{eq:GramMatrixNondegLatticeOfCr}
\begin{pmatrix}
0 & 2^r \\
2^r & 0
\end{pmatrix}
\end{equation}
is isometrically isomorphic to $\underline{C}_{2^r}$. Moreover, there is no such lattice of smaller rank.
\end{enumerate}
\end{thm}

\begin{proof}
We omit the proof of the statement in Part 2, and focus on proving the statement in Part 1.

First of all, there is no even nondegenerate lattice of rank $1$ whose discriminant module is isomorphic to $\underline{B}_{2^r}$, since such discriminant module must be a cyclic group, while $\underline{B}_{2^r}$ is not. Secondly, it can be verified by Condition 2 of Lemma \ref{lemm:HowToFindEvenLatticeForD2r} that there is also no such lattice of rank $2$. And then we shall prove there is no such lattice of rank $3$ by contradiction. Assume that there is a lattice with such property. Then its Gram matrix with respect to some basis must be (by Lemma \ref{lemm:HowToFindEvenLatticeForD2r})
\begin{equation*}
\begin{pmatrix}
2^r & 0 & 0 \\
0 & 2^r & 0 \\
0 & 0 & 1
\end{pmatrix}
\cdot
\begin{pmatrix}
2x & y & v \\
y & 2z & w \\
2^rv & 2^rw & 2t
\end{pmatrix}
\end{equation*}
where $x,y,z,v,w,t$ are all integers such that $x,y,z$ are odd, and the determinant of the second matrix in the above expression is $\pm 1$. But we see that the determinant must be even by expansion with respect the third row, which is a contradiction. Finally, we shall verify that  the given matrix in Part 1 of the theorem actually provides an even nondegenerate lattice whose discriminant module is isomorphic to $\underline{B}_{2^r}$. Note that this matrix is a product of $\mathop{\mathrm{diag}}(2^r,2^r,1,1)$ and
\begin{equation}
S =
\begin{pmatrix}
2 & 1 & 0 & 0 \\
1 & 2 & v & 0 \\
0 & 2^rv & 2x & 1 \\
0 & 0 & 1 & \varepsilon \cdot 2
\end{pmatrix}
,
\end{equation}
and that $v$ exists and $x$ is an integer. By direct calculation we see that $\det(S)=1$, and $S$ satisfies the conditions stated in Lemma \ref{lemm:HowToFindEvenLatticeForD2r}. Hence the discriminant module of this lattice is isomorphic to $\underline{B}_{2^r}$ by Lemma \ref{lemm:HowToFindEvenLatticeForD2r}, which concludes the proof.
\end{proof}

\begin{coro}
\label{coro:FQMisDicsModOfEvenLattice}
Any finite quadratic module is isometrically isomorphic to the discriminant module of some even nondegenerate lattice.
\end{coro}
\begin{proof}
This follows immediately from Theorem \ref{thm:JordanDecompositionOfFQM}, Theorem \ref{thm:LatticeOfLeastRankOfFQMApr}, Theorem \ref{thm:LatticeOfLeastRankOfFQMA2r}, and Theorem \ref{thm:LatticeOfLeastRankOfFQMBrCr}.
\end{proof}

\begin{examp}
We look for lattices whose discriminant module is isometrically isomorphic to $\underline{A}_{m}$ defined in \eqref{eq:DefAm} for $m$ being an odd integer greater than $1$. Suppose that $m=p_1^{r_1}p_2^{r_2}\dots p_k^{r_k}$ with $p_1, p_2,\dots,p_k$ different odd primes, and $r_1,r_2,\dots, r_n$ positive integers. Set $m_i=m/p_i^{r_i}$, then $m_i$'s are relatively prime. Find integers $M_i$ such that $m_1M_1+\dots m_kM_k = 1$. It can be verified that $\underline{A}_{m}$ is isometrically isomorphic to the orthogonal direct sum of the finite quadratic modules $\underline{A}_{p_i^{r_i}}^{M_i}$ with $1 \leq i \leq k$, and the isomorphism is given by the map $x+m\mathbb{Z} \mapsto (x+p_1^{r_1}\mathbb{Z}, x+p_2^{r_2}\mathbb{Z}, \dots, x+p_k^{r_k}\mathbb{Z})$. To find an even nondegenerate lattice with discriminant module $\underline{A}_{m}$, we first construct lattices with Gram matrix \eqref{eq:GramMatrixNondegLatticeOfApr1}, \eqref{eq:GramMatrixNondegLatticeOfApr2}, \eqref{eq:GramMatrixNondegLatticeOfApr3}, or \eqref{eq:GramMatrixNondegLatticeOfApr4} for $\underline{A}_{p_i^{r_i}}^{M_i}$, and then form there orthogonal direct sum. Specially, if $m$ is a square, then $\legendre{M_i}{p_i}=\legendre{m_i}{p_i}=1$. Hence only the Gram matrix \eqref{eq:GramMatrixNondegLatticeOfApr1} would occur. As a consequence, the even nondegenerate lattice with Gram matrix
\begin{equation}
\mathop{\mathrm{diag}}\left(
\begin{pmatrix}
2p_1^{r_1} & p_1^{r_1} \\
p_1^{r_1} & \frac{1}{2}\left(p_1^{r_1}-\legendre{-1}{p_1}^{r_1}\right)
\end{pmatrix}
\dots,
\begin{pmatrix}
2p_k^{r_k} & p_k^{r_k} \\
p_k^{r_k} & \frac{1}{2}\left(p_k^{r_k}-\legendre{-1}{p_k}^{r_k}\right)
\end{pmatrix}
\right)
\end{equation}
has discriminant module $\underline{A}_{m}$. The rank of this lattice is $2k$, which is not necessarily the least rank of such lattices.
\end{examp}

\section{Discriminant modules of positive definite even lattices}
\label{sec:Discriminant modules of positive definite even lattices}
We have proved the theorem of Wall (Corollary \ref{coro:FQMisDicsModOfEvenLattice}) that any finite quadatic module is a discriminant module of an even nondegenerate lattice, by giving concrete constructions for those indecomposable modules. In this section, we aim to find positive definite even lattices for arbitrary finite quadratic modules.

To be aware of the possible least rank before constructing such lattices, we shall need the Milgram's formula. Before stating the theorem, we recall some notions. Let $\underline{M}=(M,Q)$ be a finite quadratic module, we define the $\sigma$-invariant of $\underline{M}$ as
\begin{equation}
\label{deff:SigmaInvariant}
\sigma(\underline{M})=\frac{1}{\sqrt{\vert M \vert}}\sum_{x \in M}\exp(-2\uppi\mathrm{i}Q(x)).
\end{equation}
For an even nondegenerate lattice $\underline{L}=(L,B)$, the $\sigma$-invariant of $\underline{L}$ is defined to be the $\sigma$-invariant of $D_{\underline{L}}$. And we define the signature of $\underline{L}$ (denoted by $\mathop{\mathrm{sign}}(\underline{L})$) to be the signature of the real orthogonal geometry $(\mathbb{R} \otimes L, B)$, i.e., the positive index of inertia minus the negative index of inertia of $(\mathbb{R} \otimes L, B)$. These two quantities are related by
\begin{thm}
\label{thm:MilgramTheorem}
Let $\underline{L}=(L,B)$ be an even nondegenerate lattice, and $D_{\underline{L}}$ be its discriminant module.  Then
\begin{equation}
\sigma(\underline{L})=\sigma(D_{\underline{L}})=\exp\left(-2\uppi\mathrm{i}\left(\frac{\mathop{\mathrm{sign}}(\underline{L})}{8}\right)\right).
\end{equation}
\end{thm}
This identity, which is essentially on evaluation of Gauss sums, was first discovered and proved by Milgram. An easier proof (comparing to the original proof of Milgram) can be founed in \cite[Appendix 4]{MH1973}, where the authors show that the $\sigma$-invariant actually depends only on the ambient rational inner product space, and is multiplicative with respect to the orthogonal direct sum of rational inner product spaces, and then they prove the identity  by evaluating a special Gauss sum using one-dimensional Fourier analysis.

Thus, we can evaluate the $\sigma$-invariant of  any finite quadratic module $\underline{M}$ by first finding an even nondegenetate lattice whose discriminant module is just $\underline{M}$, and then using Milgram's formula. One can also evaluate $\sigma(\underline{M})$ directly, as Str\"omberg has worked out this in \cite[Section 3]{Stromberg2013}. We list these results for the reader's convenience.
\begin{prop}
\label{prop:SigmaInvariantForFQM}
Let $p$ be a prime, and $a$ an integer with $(p,a)=1$, and let $r$ be a positive integer. Then
\begin{align*}
\sigma(\underline{A}_{p^r}^a) &= \legendre{2a}{p^r}\exp\left(-2\uppi\mathrm{i}\frac{1-p^r}{8}\right) \quad \text{for } p > 2, \\
\sigma(\underline{A}_{2^r}^a) &= \legendre{a}{2^r}\exp\left(-2\uppi\mathrm{i}\frac{a}{8}\right) \quad \text{for } p = 2, \\
\sigma(\underline{B}_{2^r}) &=  (-1)^r, \\
\sigma(\underline{C}_{2^r}) &=  1.
\end{align*}
\end{prop}
As has been pointed out, these identities can be deduced by evaluating the inertia of the matrices \eqref{eq:GramMatrixNondegLatticeOfApr1}, \eqref{eq:GramMatrixNondegLatticeOfApr2}, \eqref{eq:GramMatrixNondegLatticeOfApr3}, \eqref{eq:GramMatrixNondegLatticeOfApr4}, \eqref{eq:GramMatrixNondegLatticeOfA2r2}, \eqref{eq:GramMatrixNondegLatticeOfA2r3}, \eqref{eq:GramMatrixNondegLatticeOfBr}, \eqref{eq:GramMatrixNondegLatticeOfCr}, and using Theorem \ref{thm:MilgramTheorem} (All inertia can be worked out by standard procedure of completing squares over the reals for these matrices, so we omit the tedious details.).

\begin{rema}
\label{rema:CompletionOfThmApr}
In Part 4 of Theorem \ref{thm:LatticeOfLeastRankOfFQMApr}, we have constructed an even nondegenerate lattice of rank $4$ whose discriminant module is $\underline{A}_{p^r}^a$, under certain conditions. Also we have showed that there is no such lattice of rank $1$ and $2$. We can now assert: the lattice constructed there is of the least rank. This can be proved by the formula of $\sigma$-invariants just stated in Proposition \ref{prop:SigmaInvariantForFQM}.
\end{rema}

We now turn to the task of finding positive definite even lattices of the least positive rank with given finite quadratic modules.
\begin{lemm}
\label{lemm:rankMod8OfPositive DefLattice}
Let $\underline{L}=(L,B_L)$ be a positive definite even lattice of rank $n \in \mathbb{Z}_{\geq 1}$, and $D_{\underline{L}}$ be its discriminant module. Let $\underline{M}=(M,Q_M)$ be a finite quadratic module. Let $p$ be a prime and $a$ be an integer with $(a,p)=1$, and let $r$ be a positive integer. Suppose that $D_{\underline{L}}$ and $\underline{M}$ are isometrically isomorphic.
\begin{enumerate}
\item If $\underline{M}$ is trivial, i.e., $\vert M \vert=1$, then $n \equiv 0 \mod 8$.
\item If $\underline{M}=\underline{A}_{p^r}^a$, with $r$ even and $p > 2$, then $n \equiv 0 \mod 8$.
\item If $\underline{M}=\underline{A}_{p^r}^a$, with $r$ odd and $p > 2$, then $n \equiv 3 - \legendre{-1}{p} - 2\legendre{a}{p} \mod 8$.
\item If $\underline{M}=\underline{A}_{2^r}^a$, with $r$ even and $(a,2)=1$, then $n \equiv a \mod 8$.
\item If $\underline{M}=\underline{A}_{2^r}^a$, with $r$ odd and $(a,2)=1$, then $n \equiv a + 2 - 2\legendre{a}{2} \mod 8$.
\item If $\underline{M}=\underline{B}_{2^r}$, then $n \equiv 2 - 2(-1)^r \mod 8$.
\item If $\underline{M}=\underline{C}_{2^r}$, then $n \equiv 0 \mod 8$.
\end{enumerate}
\end{lemm}
\begin{proof}
This follows immediately from Theorem \ref{thm:MilgramTheorem} and Proposition \ref{prop:SigmaInvariantForFQM}.
\end{proof}

\begin{rema}
Let $\underline{L}=(L,B)$ be an even nondegenerate lattice. If $D_{\underline{L}}$ is the zero module, i.e., $L=L^{\sharp}$, then $\underline{L}$ is called a unimodular lattice. The name comes from the fact that the volumn of a closed fundamental domain of $(\mathbb{R} \otimes L) / L$ in the real space $\mathbb{R} \otimes L$ is $1$ under a natural selection of Lebesgue measure. The investigation of  this kind of lattices has a long history and constitutes an important area in number theory. In the indefinite case, one has a perfect structure theorem, while in the definite case, one has related the theory on such lattices to theories on many other disciplines, such as packing of equal balls in Euclidean spaces. The answer to the question of finding a positive definite even lattice with the zero discriminant module and least positive rank is well-known: The lattice $\Gamma_8 := \left\{(x_1,\dots, x_8) \in \mathbb{R}^8 \middle\vert x_1, \dots, x_8\text{ are all in } \mathbb{Z}\text{ or }1/2+\mathbb{Z},\,x_1+\dots + x_8 \in 2\mathbb{Z}\right\}$ satisfies the required conditions. See \cite{CS1999} for a complete discussion. In the rest of this section, we focus on finding positive definite even lattices for indecomposable (nonzero) finite quadratic modules.
\end{rema}

Now we proceed to deal with the modules $\underline{A}_{p^r}^a$ with $r$ even.
\begin{thm}
\label{thm:theLeastRankPosDefLatticeForAprEven}
Let $p$ be an odd prime, and $a$ be an integer with $(a,p)=1$, let $r$ be an even positive integer.
\begin{enumerate}
\item If $p \neq 7$, then the discriminant module of the positive definite even lattice given by the following Gram matrix
\begin{equation}
\label{eq:matrixForAprPosDefpneq7rEven}
\begin{pmatrix}
2qp^r & p^rv & 0 & 0 & 0 & 0 & 0 & 0 \\
p^rv & 2x & 1 & 0 & 0 & 0 & 0 & 0 \\
0 & 1 & 2 & 1 & 0 & 0 & 0 & 0 \\
0 & 0 & 1 & 2 & 1 & 0 & 0 & 0 \\
0 & 0 & 0 & 1 & 2 & 1 & 0 & 0 \\
0 & 0 & 0 & 0 & 1 & 2 & 1 & 0 \\
0 & 0 & 0 & 0 & 0 & 1 & 2 & 1 \\
0 & 0 & 0 & 0 & 0 & 0 & 1 & 2
\end{pmatrix}
\end{equation}
is isometrically isomorphic to $\underline{A}_{p^r}^a$, where $q \equiv 4 \mod 7 $ is a prime such that $\legendre{q}{p}=\legendre{a}{p}$, $v$ is a solution of the congruence $7p^rv^2 \equiv -1 \mod 4q$, and $x=\frac{7p^rv^2+12q+1}{28q}$.
\item If $p = 7$, and $\legendre{a}{7}=1$, then the discriminant module of the positive definite even lattice given by the following Gram matrix
\begin{equation}
\label{eq:matrixForAprPosDefpeq7rEvenLegendre1}
\begin{pmatrix}
8\cdot 7^r & 7^rv & 0 & 0 & 0 & 0 & 0 & 0 \\
7^rv & 2x & 1 & 0 & 0 & 0 & 0 & 0 \\
0 & 1 & 2 & 1 & 0 & 0 & 0 & 0 \\
0 & 0 & 1 & 2 & 1 & 0 & 0 & 0 \\
0 & 0 & 0 & 1 & 2 & 1 & 0 & 0 \\
0 & 0 & 0 & 0 & 1 & 2 & 1 & 0 \\
0 & 0 & 0 & 0 & 0 & 1 & 2 & 1 \\
0 & 0 & 0 & 0 & 0 & 0 & 1 & 2
\end{pmatrix}
\end{equation}
is isometrically isomorphic to $\underline{A}_{7^r}^a$, where $v$ is a solution of the congruence $7^{r-1}v^2 \equiv -1 \mod 16$, and $x=\frac{7^rv^2+7}{16}$.
\item If $p = 7$, and $\legendre{a}{7}=-1$, then the discriminant module of the positive definite even lattice given by the following Gram matrix
\begin{equation}
\label{eq:matrixForAprPosDefpeq7rEvenLegendre-1}
\begin{pmatrix}
10\cdot 7^r & 7^rv & 0 & 0 & 0 & 0 & 0 & 0 \\
7^rv & 2x & 5 & 0 & 0 & 0 & 0 & 0 \\
0 & 5 & 4 & 1 & 0 & 0 & 0 & 0 \\
0 & 0 & 1 & 2 & 1 & 0 & 0 & 0 \\
0 & 0 & 0 & 1 & 2 & 1 & 0 & 0 \\
0 & 0 & 0 & 0 & 1 & 2 & 1 & 0 \\
0 & 0 & 0 & 0 & 0 & 1 & 2 & 1 \\
0 & 0 & 0 & 0 & 0 & 0 & 1 & 2
\end{pmatrix}
\end{equation}
is isometrically isomorphic to $\underline{A}_{7^r}^a$, where $v$ is a solution of the congruences $7^rv^2 \equiv 1 \mod 20$, and $x=\frac{7^rv^2+79}{20}$.
\end{enumerate}
Moreover, in each of the above cases, $8$ is the least rank of positive definite even lattices whose discriminant module is isomorphic to $\underline{A}_{p^r}^a$.
\end{thm}

\begin{proof}
Before proving the theorem itself, we state an auxiliary property. We define a sequence of matrices $F_l$ for $l \in \mathbb{Z}_{\geq 1}$ as follows. The matrix $F_l$ is a matrix of size $l \times l$, and the $(i,j)$-entry $f_{ij}$ is equal to $2$, if $i=j$, is equal to $1$, if $i-j=\pm 1$, and is equal to $0$ otherwise. One shows by induction that $\det(F_l)=l+1$. In particular, $\det(F_5)=6$, and $\det(F_6)=7$. Also note that $F_l$'s are all positive definite matrices.

Now proceed to prove the theorem. First of all, we can prove that $v$ exists in three cases by elementary number theory, and hence $x$ is an integer. Secondly, divide the first row of each of the matrices \eqref{eq:matrixForAprPosDefpneq7rEven}, \eqref{eq:matrixForAprPosDefpeq7rEvenLegendre1}, and \eqref{eq:matrixForAprPosDefpeq7rEvenLegendre-1} by $p^r$, and denote the result by $S$. Then it can be verified that $S$ is an integral matrix and satisfies the three conditions posed in Lemma \ref{lemm:HowToFindEvenLatticeForApr}. Thus, to conclude that the discriminant modules of these lattices are isomorphic to $\underline{A}_{p^r}^a$, it remains to prove that $\det(S)=\pm 1$ in three cases. Using the determinant of $F_5$ and $F_6$ just stated, and using the expansion according to the first row, we find the determinant of the matrix $S$ obtaining from \eqref{eq:matrixForAprPosDefpneq7rEven} is $28qx-7p^rv^2-12q$, which equals $1$ since $x=\frac{7p^rv^2+12q+1}{28q}$. Similarly the matrices $S$ obtaining from \eqref{eq:matrixForAprPosDefpeq7rEvenLegendre1} and \eqref{eq:matrixForAprPosDefpeq7rEvenLegendre-1} also have determinant $1$. Thirdly, the lattices given by these Gram matrices are positive definite. To prove this, note that by Theorem \ref{thm:MilgramTheorem} and Proposition \ref{prop:SigmaInvariantForFQM} the signatures are $0$ or $\pm 8$. But in all three cases, each of the ambient real orthogonal geometries has a positive definite subspace of dimension $5$ (with Gram matrix $F_5$), hence $8$ is the only possible signature, i.e., these lattices are positive definite. This assertion can also be proved by completing squares directly. Finally, we see that an even nondegenerate nonzero lattice of rank less than $8$ can not be positive definite by Part 2 of Lemma \ref{lemm:rankMod8OfPositive DefLattice}.
\end{proof}

Next we consider $\underline{A}_{p^r}^a$ with $r$ odd.
\begin{thm}
\label{thm:theLeastRankPosDefLatticeForAprOdd}
Let $p$ be an odd prime, and $a$ be an integer with $(a,p)=1$, let $r$ be an odd positive integer.
\begin{enumerate}
\item If $\legendre{-1}{p}=1$ and $\legendre{a}{p}=1$, then let $q$, $v$, and $x$ be as in Part 1 of Theorem \ref{thm:theLeastRankPosDefLatticeForAprEven}, and we have the discriminant module of the positive definite even lattice given by the Gram matrix \eqref{eq:matrixForAprPosDefpneq7rEven} is isometrically isomorphic to $\underline{A}_{p^r}^a$.
\item If $\legendre{-1}{p}=1$ and $\legendre{a}{p}=-1$, then let $q$, $v$, and $x$ be as in Part 4 of Theorem \ref{thm:LatticeOfLeastRankOfFQMApr}, and we have the discriminant module of the positive definite even lattice given by the Gram matrix \eqref{eq:GramMatrixNondegLatticeOfApr4} is isometrically isomorphic to $\underline{A}_{p^r}^a$.
\item If $\legendre{-1}{p}=-1$ and $\legendre{a}{p}=1$, then the discriminant module of the positive definite even lattice given by the Gram matrix \eqref{eq:GramMatrixNondegLatticeOfApr1} in Part 1 of Theorem \ref{thm:LatticeOfLeastRankOfFQMApr} is isometrically isomorphic to $\underline{A}_{p^r}^a$.
\item If $\legendre{-1}{p}=-1$ and $\legendre{a}{p}=-1$,  then the discriminant module of the positive definite even lattice given by the Gram matrix
\begin{equation}
\label{eq:matrixForAprPosDefrOddLegendre-1-1}
\begin{pmatrix}
2qp^r & p^rv & 0 & 0 & 0 & 0 \\
p^rv & 2x & 1 & 0 & 0 & 0 \\
0 & 1 & 2 & 1 & 0 & 0 \\
0 & 0 & 1 & 2 & 1 & 0 \\
0 & 0 & 0 & 1 & 2 & 1 \\
0 & 0 & 0 & 0 & 1 & 2
\end{pmatrix}
\end{equation}
is isometrically isomorphic to $\underline{A}_{p^r}^a$, where $q \equiv 3 \mod 5$ is a prime such that $\legendre{q}{p}=-1$, and $v$ is a solution of the congruence $5p^rv^2 \equiv -1 \mod 4q$, and $x=\frac{5p^rv^2+8q+1}{20q}$.
\end{enumerate}
Moreover, in each of the above cases,  the rank of the given lattice is the least among all positive definite even lattices whose discriminant module is isomorphic to $\underline{A}_{p^r}^a$.
\end{thm}

\begin{proof}
We have proved the assertions of Part 1, Part 2, and Part 3 in Theorem \ref{thm:LatticeOfLeastRankOfFQMApr} and Theorem \ref{thm:theLeastRankPosDefLatticeForAprEven}, except for the fact that the lattices given by the Gram matrices in Part 2 and Part 3 are positive definite, which can be verified directly.

To prove the statement in Part 4, we use Lemma \ref{lemm:HowToFindEvenLatticeForApr}, as in the proof of Theorem \ref{thm:theLeastRankPosDefLatticeForAprEven}. We omit the details.

Finally, from Part 3 of Lemma \ref{lemm:rankMod8OfPositive DefLattice}, we see that the rank of the given lattice is actually the least among all positive definite even lattices whose discriminant module is isomorphic to $\underline{A}_{p^r}^a$ in each of the four cases.
\end{proof}

We now turn to the modules $\underline{A}_{2^r}^a$ with $r$ even.
\begin{thm}
\label{thm:theLeastRankPosDefLatticeForA2rEven}
Let $r$ be an even positive integer, and $a$ an odd integer.
\begin{enumerate}
\item If $a \equiv 1 \mod 8$, then the discriminant module of the positive definite even lattice given by the $1 \times 1$ Gram matrix $(2^r)$ is isometrically isomorphic to $\underline{A}_{2^r}^a$.
\item If $a \equiv 3 \mod 8$, then the discriminant module of the positive definite even lattice given by the Gram matrix \eqref{eq:GramMatrixNondegLatticeOfA2r3} with $\varepsilon = -1$ in Part 3 of Theorem \ref{thm:LatticeOfLeastRankOfFQMA2r} is isometrically isomorphic to $\underline{A}_{2^r}^a$.
\item If $a \equiv 5 \mod 8$, then the discriminant module of the positive definite even lattice given by the Gram matrix
\begin{equation}
\label{eq:matrixForA2rPosDefrEven5Mod8}
\begin{pmatrix}
5\cdot 2^r & 2^rv & 0 & 0 & 0 \\
2^rv & 2x & 1 & 0 & 0 \\
0 & 1 & 2 & 1 & 0 \\
0 & 0 & 1 & 2 & 1 \\
0 & 0 & 0 & 1 & 2
\end{pmatrix}
\end{equation}
is isometrically isomorphic to $\underline{A}_{2^r}^a$, where $v$ is a solution of the congruence $2^{r-1}v^2 \equiv -2 \mod 5$, and $x=\frac{2^{r-1}v^2+2}{5}$.
\item If $a \equiv 7 \mod 8$, then the discriminant module of the positive definite even lattice given by the Gram matrix
\begin{equation}
\label{eq:matrixForA2rPosDefrEven5Mod8}
\begin{pmatrix}
7\cdot 2^r & 2^rv & 0 & 0 & 0 & 0 & 0 \\
2^rv & 2x & 1 & 0 & 0 & 0 & 0 \\
0 & 1 & 2 & 1 & 0 & 0 & 0\\
0 & 0 & 1 & 2 & 1 & 0 & 0 \\
0 & 0 & 0 & 1 & 2 & 1 & 0 \\
0 & 0 & 0 & 0 & 1 & 2 & 1 \\
0 & 0 & 0 & 0 & 0 & 1 & 2
\end{pmatrix}
\end{equation}
is isometrically isomorphic to $\underline{A}_{2^r}^a$, where $v$ is a solution of the congruence $2^{r-1}v^2 \equiv -3 \mod 7$, and $x = \frac{2^{r-1}v^2+3}{7}$.
\end{enumerate}
Moreover, in each of the above cases,  the rank of the given lattice is the least among all positive definite even lattices whose discriminant module is isomorphic to $\underline{A}_{2^r}^a$.
\end{thm}

\begin{proof}
The statements in Part 1 and Part 2 have been proved in Theorem \ref{thm:LatticeOfLeastRankOfFQMA2r}, except for the fact that the lattices defined in these two cases are positive definite, which can be verified directly (e.g., by completing squares). To prove statements in Part 3 and Part 4, we argue as in the proof of Theorem \ref{thm:theLeastRankPosDefLatticeForAprEven}, using Lemma \ref{lemm:HowToFindEvenLatticeForA2r} instead of Lemma \ref{lemm:HowToFindEvenLatticeForApr}. Finally, to prove the last assertion in the theorem, use Part 4 of Lemma \ref{lemm:rankMod8OfPositive DefLattice}.
\end{proof}

Now turn to the modules $\underline{A}_{2^r}^a$ with $r$ odd.
\begin{thm}
\label{thm:theLeastRankPosDefLatticeForA2rOdd}
Let $r$ be an odd integer greater than $1$, and $a$ be an odd integer.
\begin{enumerate}
\item If $a \equiv 1 \mod 8$, then the conclusion is the same as in Part 1 of Theorem \ref{thm:theLeastRankPosDefLatticeForA2rEven}.
\item If $a \equiv 3 \mod 8$, then the discriminant module of the positive definite even lattice given by the Gram matrix
\begin{equation}
\label{eq:matrixForA2rPosDefrOdd3Mod8}
\begin{pmatrix}
19\cdot 2^r & 2^rv & 0 & 0 & 0 & 0 & 0 \\
2^rv & 2x & 1 & 0 & 0 & 0 & 0 \\
0 & 1 & 2 & 1 & 0 & 0 & 0\\
0 & 0 & 1 & 2 & 1 & 0 & 0 \\
0 & 0 & 0 & 1 & 2 & 1 & 0 \\
0 & 0 & 0 & 0 & 1 & 2 & 1 \\
0 & 0 & 0 & 0 & 0 & 1 & 2
\end{pmatrix}
\end{equation}
is isometrically isomorphic to $\underline{A}_{2^r}^a$, where $v$ is a solution of the congruence $2^{r-1}v^2 \equiv -8 \mod 19$, and $x = \frac{2^{r-1}v^2+8}{19}$.
\item If $a \equiv 5 \mod 8$, then the discriminant module of the positive definite even lattice given by the Gram matrix
\begin{equation}
\label{eq:matrixForA2rPosDefrOdd5Mod8}
\begin{pmatrix}
5 \cdot 2^r & 2^rv & 0 & 0 & 0 & 0 & 0 & 0& 0 \\
2^rv & 2x & 7 & 0 & 0 & 0 & 0 & 0 & 0 \\
0 & 7 & 4 & 1 & 0 & 0 & 0 & 0 & 0 \\
0 & 0 & 1 & 2 & 1 & 0 & 0 & 0 & 0 \\
0 & 0 & 0 & 1 & 2 & 1 & 0 & 0 & 0 \\
0 & 0 & 0 & 0 & 1 & 2 & 1 & 0 & 0 \\
0 & 0 & 0 & 0 & 0 & 1 & 2 & 1 & 0 \\
0 & 0 & 0 & 0 & 0 & 0 & 1 & 2 & 1 \\
0 & 0 & 0 & 0 & 0 & 0 & 0 & 1 & 2 
\end{pmatrix}
\end{equation}
is isometrically isomorphic to $\underline{A}_{2^r}^a$, where $v$ is a solution of the congruence $2^{r-1}v^2 \equiv 1 \mod 5$, and $x=\frac{2^{r-1}v^2+39}{5}$.
\item If $a \equiv 7 \mod 8$, then the conclusion is the same as in Part 4 of Theorem \ref{thm:theLeastRankPosDefLatticeForA2rEven}.
\end{enumerate}
Moreover, in each of the above cases,  the rank of the given lattice is the least among all positive definite even lattices whose discriminant module is isomorphic to $\underline{A}_{2^r}^a$.
\end{thm}

\begin{proof}
We omit the proof, since it is similar to that of the previous theorem.
\end{proof}

Finally, we shall deal with the modules $\underline{B}_{2^r}$ and $\underline{C}_{2^r}$.
\begin{thm}
\label{thm:theLeastRankPosDefLatticeForB2rC2r}
Let $r$ be a positive integer.
\begin{enumerate}
\item If $r$ is odd, then the conclusion for $\underline{B}_{2^r}$ is the same as in Part 1 of Theorem \ref{thm:LatticeOfLeastRankOfFQMBrCr} with $\varepsilon=1$. The lattice defined there is actually positive definite.
\item If $r$ is even, then the discriminant module of the positive definite even lattice given by the Gram matrix
\begin{equation}
\label{eq:matrixForB2rPosDefrEven}
\begin{pmatrix}
2 \cdot 2^r & 2^r & 0 & 0 & 0 & 0 & 0 & 0 \\
2^r & 10 \cdot 2^r & 2^rv & 0 & 0 & 0 & 0 & 0 \\
0 & 2^rv & 2x & 1 & 0 & 0 & 0 & 0 \\
0 & 0 & 1 & 2 & 1 & 0 & 0 & 0 \\
0 & 0 & 0 & 1 & 2 & 1 & 0 & 0 \\
0 & 0 & 0 & 0 & 1 & 2 & 1 & 0 \\
0 & 0 & 0 & 0 & 0 & 1 & 2 & 1 \\
0 & 0 & 0 & 0 & 0 & 0 & 1 & 2
\end{pmatrix}
\end{equation}
is isometrically isomorphic to $\underline{B}_{2^r}$, where $v$ is a solution of the congruence $2^rv^2 \equiv -8 \mod 19$, and $x = \frac{2^rv^2+8}{19}$.
\item The discriminant module of the positive definite even lattice given by the Gram matrix
\begin{equation}
\label{eq:matrixForC2rPosDefrEven}
\begin{pmatrix}
2 \cdot 2^r & 2^r & 0 & 0 & 0 & 0 & 0 & 0 \\
2^r & 4 \cdot 2^r & 2^rv & 0 & 0 & 0 & 0 & 0 \\
0 & 2^rv & 2x & 1 & 0 & 0 & 0 & 0 \\
0 & 0 & 1 & 2 & 1 & 0 & 0 & 0 \\
0 & 0 & 0 & 1 & 2 & 1 & 0 & 0 \\
0 & 0 & 0 & 0 & 1 & 2 & 1 & 0 \\
0 & 0 & 0 & 0 & 0 & 1 & 2 & 1 \\
0 & 0 & 0 & 0 & 0 & 0 & 1 & 2
\end{pmatrix}
\end{equation}
is isometrically isomorphic to $\underline{C}_{2^r}$, where $v$ is a solution of the congruence $2^rv^2 \equiv -3 \mod 7$, and $x = \frac{2^rv^2+3}{7}$.
\end{enumerate}
Moreover, in each of the above cases,  the rank of the given lattice is the least among all positive definite even lattices whose discriminant module is isomorphic to $\underline{B}_{2^r}$ (in Case 1 and 2) or $\underline{C}_{2^r}$ (in Case 3).
\end{thm}

\begin{proof}
To prove the statement in Part 1, we need only to vefiry that the lattice mentioned is positive definite, which is obtained by completing squares. To prove statements in another two parts, first note that the congruences in both parts are solvable, so the quantity $x$ is integer in both cases. Then use Lemma \ref{lemm:HowToFindEvenLatticeForD2r} to conclude that matrices \eqref{eq:matrixForB2rPosDefrEven} and \eqref{eq:matrixForC2rPosDefrEven} give lattices whose discriminant modules are $\underline{B}_{2^r}$ and $\underline{C}_{2^r}$ respectively.
\end{proof}

As in Corollary \ref{coro:FQMisDicsModOfEvenLattice}, we can combine the constructions in this section to obtain a theorem about arbitrary finite quadratic modules.
\begin{coro}
\label{coro:FQMisDicsModOfPosDefEvenLattice}
Any finite quadratic module is isometrically isomorphic to the discriminant module of some positive definite even lattice.
\end{coro}
\begin{proof}
We proceed similarly as in the proof of Corollary \ref{coro:FQMisDicsModOfEvenLattice}, using Theorem \ref{thm:JordanDecompositionOfFQM}, Theorem \ref{thm:theLeastRankPosDefLatticeForAprEven}, Theorem \ref{thm:theLeastRankPosDefLatticeForAprOdd}, Theorem \ref{thm:theLeastRankPosDefLatticeForA2rEven}, Theorem \ref{thm:theLeastRankPosDefLatticeForA2rOdd}, and Theorem \ref{thm:theLeastRankPosDefLatticeForB2rC2r} instead.
\end{proof}

\section{Some Remarks}
\label{sec:Some Remarks}
\paragraph{Finite symmetric bilinear modules.} There is a parallel theory of finite symmetric bilinear modules, besides that of finite quadratic modules. Here a finite symmetric bilinear module means a finite $\mathbb{Z}$-module equipped with a nondegenerate symmetric bilinear form taking values in $\mathbb{Q}/\mathbb{Z}$. All of the methods described in Section \ref{sec:Billinear map modules and quadratic map modules} also apply to such structures. Analogous versions of Corollary \ref{coro:FQMStroNondeg} and Proposition \ref{prop:HomIsoQuoAndDoubleOrtho} remain true for finite symmetric bilinear modules. The indecomposable ones still have the form like those in Definition \ref{deff:FQMABC}, that is (by abuse of language),
\begin{align*}
\underline{A}_{p^r}^a &:= \left(\mathbb{Z}/p^r\mathbb{Z},\,(\overline{x}, \overline{y}) \mapsto \frac{a}{p^r}xy+\mathbb{Z}\right) \, \text{for }p>2, \\
\underline{A}_{2^r}^a &:= \left(\mathbb{Z}/2^r\mathbb{Z},\,(\overline{x}, \overline{y}) \mapsto \frac{a}{2^{r}}xy+\mathbb{Z}\right) \, \text{for } p=2, \\
\underline{B}_{2^r}   &:= \left((\mathbb{Z}/2^r\mathbb{Z})^2,\,\left((\overline{x_1}, \overline{y_1}),(\overline{x_2}, \overline{y_2})\right) \mapsto \frac{2x_1x_2+x_1y_2+x_2y_1+2y_1y_2}{2^r}+\mathbb{Z} \right), \\
\underline{C}_{2^r}   &:= \left((\mathbb{Z}/2^r\mathbb{Z})^2,\,\left((\overline{x_1}, \overline{y_1}),(\overline{x_2}, \overline{y_2})\right) \mapsto \frac{x_1y_2+x_2y_1}{2^r}+\mathbb{Z}\right),
\end{align*}
where $p$ is a prime, $a$ is an integer not divisible by $p$, and $\overline{x}$ denotes $x+p^r\mathbb{Z}$. One also has a structure theorem, which asserts that any finite symmetric bilinear module can be decomposed as an orthogonal direct sum of the modules just listed. See \cite[\S 5]{Wall1963} for a proof and further discussion. We claim that the methods using in our proof of Theorem \ref{thm:JordanDecompositionOfFQM} also apply to the finite symmetric bilinear module case, with some minor adaptation.

We would like to explain why we have focused on finite quadratic modules instead of bilinear ones. Let $\underline{L}=(L,B)$ be an integral lattice, and $L^\sharp$ be its dual (c.f. Definition \ref{deff:DefOfDualLattice}). Suppose $L$ is odd, i.e., there is a vector $v \in L$ such that $Q(v) = 2^{-1}B(v,v) \in 2^{-1}+\mathbb{Z}$. Then the construction of discriminant modules in Example \ref{examp:DiscModule} is problematic, since the quantity $\underline{Q}(v + L)$, where $\underline{Q}$ denotes the quadratic map on $L^\sharp/L$, depends on the choice of a representative of the coset $v + L$. Nevertheless, the bilinear map $\underline{B}(x+L,y+L) = B(x,y)+\mathbb{Z}$ is well-defined on $L^\sharp/L$, no matter whether $L$ is even or odd. Futhermore, the structure $(L^\sharp/L, \underline{B})$ is a finite symmetric bilinear module. Thus, it seems more nature to study integral lattices using ``discriminant bilinear module''. However, only to finite quadratic modules, one can associate Weil representations which are important objects as explained in Introduction. This is the reason we focus on finite quadratic modules.

\paragraph{The question of uniqueness.} We have not mentioned whether the decompositon in Theorem \ref{thm:JordanDecompositionOfFQM} is unique in some sense. Of course, the primary decompositon given in Lemma \ref{lema:PriDecomOfFQM} is unique. So the real question is: given some finite quadratic module of cardinality $p^r$, where $p$ is a prime, and $r$ is a positive integer, does the decomposition of this module into indecomposable modules unique up to isometric isomorphism? Unless the given module is itself indecomposable, the answer is always no. Wall succeeded in getting all relations among different decompositions when $p>2$, and he also get some but not all relations when $p=2$. See \cite[\S 5]{Wall1963}. Here we only give a simple example to illustrate this. Consider the $\mathbb{Z}$-module $(\mathbb{Z}/3\mathbb{Z})^2$ equipped with a quadratic map $Q(x+3\mathbb{Z},y+3\mathbb{Z})=\frac{x^2+y^2}{3}+\mathbb{Z}$. It is easy to see this can be decomposed as an orthogonal direct sum of two modules each of which is isometrically isomorphic to $\underline{A}_{3}^1$. Now let $G$ be the submodule generated by $(1+3\mathbb{Z}, 1+3\mathbb{Z})$, then G equipped with the quadratic map inherited from $(\mathbb{Z}/3\mathbb{Z})^2$ is isometrically isomorphic to $\underline{A}_{3}^2$. A direct calculation shows that $G^\perp$ is also isometrically isomorphic to $\underline{A}_{3}^2$. Hence $\mathbb{Z}$-module $(\mathbb{Z}/3\mathbb{Z})^2$ has another essentially different decomposition as an orthogonal direct sum of two copies of $\underline{A}_{3}^2$.

\section*{Acknowledgment}
\label{sec:Acknowledgment}
Thanks are due to Prof. Hai-Gang Zhou, my Ph. D. supervisor, who gave me useful suggestion, and led me into the beautiful area of modular forms and Jacobi forms. Thanks are also due to Prof. Nils-Peter Skoruppa, a top expert in the field of Jacobi forms, who taught me knowledge and skill revelant to the content of this paper.


\begin{thebibliography}{1}

\bibitem{Borcherds1998}
R. E. Borcherds.
\newblock Automorphic forms with singularities on Grassmannians.
\newblock {\em Inventiones Mathematicae}, 132: 491--562, 1998.

\bibitem{Boylan2015}
H. Boylan.
\newblock {{\em Jacobi Forms, Finite Quadratic Modules and Weil Representations over Number Fields}}.
\newblock Switzerland: Springer International Publishing, 2015.


\bibitem{CS2017}
H. Cohen, F. Str\"omberg.
\newblock {{\em  Modular Forms: A Classical Approach}}.
\newblock Graduate Studies in Mathematics, vol 179, American Mathematical Society, Providence, Rhode Island, 2017.


\bibitem{CS1999}
J. H. Conway, N. J. A. Sloane.
\newblock {{\em  Sphere packings, Lattices and Groups}}.
\newblock Fundamental Principles of Mathematical Sciences, vol 290, Third edition, Springer-Verlag, New York, 1999.

\bibitem{Gelbart1976}
S. Gelbart.
\newblock {{\em Weil’s Representation and the Spectrum of the Metaplectic Group}}.
\newblock Lecture Notes in Mathematics, vol 530, Springer, Berlin, 1976.

\bibitem{MH1973}
J. Milnor, D. Husemoller.
\newblock {{\em  Symmetric Bilinear Forms}}.
\newblock A Series of Modern Surveys in Mathematics, vol 73, Springer, New York, 1973.

\bibitem{Nikulin1980}
V. V. Nikulin.
\newblock Integral symmetric bilinear forms and some of their applications.
\newblock {\em Mathematics of the USSR-Izvestiya}, 14(1): 103--167, 1980.

\bibitem{Roman2011}
S. Roman.
\newblock {{\em Advanced Linear Algebra. 3rd ed.}}.
\newblock Graduate Texts in Mathematics, vol 135, Springer, New York, 2011.

\bibitem{Scheithauer2009}
N. R. Scheithauer.
\newblock The Weil Representation of $\mathrm{SL}_2(\mathbb{Z})$ and Some Applications.
\newblock {\em International Mathematics Research Notices}, 2009(8): 1488--1545, 2009.

\bibitem{Serre1973}
J. P. Serre.
\newblock {\em A Course in Arithmetic}.
\newblock Springer-Verlag, 1973.

\bibitem{Skoruppa2008}
N.-P. Skoruppa.
\newblock Jacobi forms of critical weight and Weil representations.
\newblock {\em Modular forms on Schiermonnikoog}, Cambridge Univ. Press, Cambridge, 239--266, 2008.

\bibitem{Skoruppa2020}
N.-P. Skoruppa.
\newblock Weil representations associated to finite quadratic modules and vector valued modular forms.
\newblock in preparation, 2020.

\bibitem{Stromberg2013}
F. Str\"omberg.
\newblock Weil representations associated with finite quadratic modules.
\newblock {\em Math. Z.}, 275: 509--527, 2013.

\bibitem{Wall1963}
C.T.C. Wall.
\newblock Quadratic forms on finite groups, and related topics.
\newblock {\em Topology}, 2(4): 281--298, 1963.

\bibitem{Weil1964}
A. Weil.
\newblock Sur certains groupes d'op\'erateurs unitaires.
\newblock {\em Acta Math.}, 111: 143--211, 1964.

\end{thebibliography}
\end{document}